\documentclass[10pt,reqno]{amsart}
\usepackage{amsmath}
\usepackage{amssymb}
\usepackage{amsthm}
\usepackage{eepic,epic}
\usepackage{epsfig}

\newlength{\defbaselineskip}
\newcommand{\setlinespacing}[1]%
           {\setlength{\baselineskip}{#1 \defbaselineskip}}

\numberwithin{equation}{section}

\newtheorem{thm}{Theorem}[section]

\newtheorem{lem}[thm]{Lemma}
\newtheorem{prop}[thm]{Proposition}

\theoremstyle{definition}

\theoremstyle{remark}
\newtheorem{rem}[thm]{Remark}
\numberwithin{equation}{section}
\theoremstyle{theorem}

\begin{document}
\title[The radius of spatial analyticity]
{On the radius of spatial analyticity for the Klein-Gordon-Schr\"{o}dinger system}

\author{Jaeseop Ahn, Jimyeong Kim and Ihyeok Seo}

\thanks{This research was supported by NRF-2022R1A2C1011312.}


\address{Department of Mathematics, Sungkyunkwan University, Suwon 16419, Republic of Korea}
\email{j.ahn@skku.edu}

\email{jimkim@skku.edu}

\email{ihseo@skku.edu}

\begin{abstract}
In this paper, we study the persistence of spatial analyticity for the solutions to the Klein-Gordon-Schr\"{o}dinger system, which describes a physical system of a nucleon field interacting with a neutral meson field, with analytic initial data. 
Unlike the case of a single nonlinear dispersive equation,
not much is known about nonlinear dispersive systems as it is harder to show the spatial analyticity of coupled equations simultaneously. 
The only results known so far are rather recent ones for the Dirac-Klein-Gordon system
which governs the physical system when the nucleon  is described by Dirac spinor fields in the case of relativistic fields. In contrast, we aim here to study the Klein-Gordon-Schr\"{o}dinger system that works in the non-relativistic regime.
It is shown that the radius of spatial analyticity of the solutions at later times obeys an algebraic lower bound as time goes to infinity.
\end{abstract}

\maketitle

\section{Introduction}\label{sec1}
In this paper we consider the Cauchy problem for the Klein-Gordon-Schr\"{o}dinger system
\begin{equation}\label{KGS}
\begin{cases}
i\partial_t u + \Delta u = - un, \quad\quad u(0,x) = u_0(x),\\
\partial^2_t n + (1-\Delta)n = |u|^2, \quad\quad (n, \partial_t n)(0,x) = (n_0,n_1)(x),
\end{cases}
\end{equation}
where $u:\mathbb{R}^{1+d}\rightarrow \mathbb{C}$ and $n:\mathbb{R}^{1+d}\rightarrow\mathbb{R}$ for $d=1,2,3$.
This system is a classical model which describes a system of a complex scalar nucleon field $u$
interacting with a neutral real scalar meson field $n$.
The mass of the meson is normalized to be $1$.

The well-posedness of this Cauchy problem with initial data in Sobolev spaces $H^s(\mathbb{R}^d)$ has been intensively studied.
The best known result is that \eqref{KGS} is globally well-posed for $(u_0,n_0,n_1)\in L^2(\mathbb{R}^d)\times H^{s}(\mathbb{R}^d)\times H^{s-1}(\mathbb{R}^d)$ where $-1/2<s<1/2$ for $d=1$, $-1/2<s<3/2$ for $d=2$ and $-1/2<s\leq1$ for $d=3$; see \cite{H3,HGHL}.
For earlier studies, we refer the reader to \cite{H, CHT, H2} and references therein.

The main purpose of the present paper is to study  spatial analyticity of the system.
We focus on the situation
where we consider a real-analytic initial data with uniform radius of analyticity $\sigma_0$,
so there is a holomorphic extension to a complex strip $S_{\sigma_0}=\{x+iy:x,y\in\mathbb{R}^d,|y_1|,|y_2|,\cdots,|y_d|<\sigma_0\}$.
Now it is natural to ask whether the initial analyticity may be continued to a solution at later time t, but with a possibly smaller and shrinking radius of analyticity $\sigma(t)$;
we would like to estimate a decay rate of $\sigma(t)$ as time $t$ tends to infinity.

This type of question was first introduced by Kato and Masuda \cite{KM} in 1986 and there are plenty of works for a single nonlinear dispersive equation such as the Kadomtsev-Petviashvili equation \cite{B}, KdV type equations \cite{BG,BGK,SS,Te2,HW,PS,AKS2}, Schr\"{o}dinger equations \cite{BGK2, Te, AKS}, and Klein-Gordon equations \cite{P}.

On the other hand, not much is known about nonlinear dispersive systems as it is harder to show the spatial analyticity of coupled equations, controlling all of them at the same time.
The only results known so far are rather recent ones for the Dirac-Klein-Gordon system \cite{ST,S}
which governs the physical system when the nucleon $u$ is described by Dirac
spinor fields in the case of relativistic fields.
In contrast, we aim here to study the Klein-Gordon-Schr\"{o}dinger system that works in the non-relativistic regime.

A class of analytic functions suitable for the problem we address here is the Gevrey class $G^{\sigma,s}(\mathbb{R}^d)$
introduced by Foias and Temam \cite{FT}, which may be defined with the norm
$$\| f\|_{G^{\sigma,s}} = \big\| e^{\sigma\| D\|}\langle D\rangle^s f \big\|_{L^2}$$
for $\sigma\geq0$ and $s\in\mathbb{R}$.
Here, $D=-i\nabla$ with Fourier symbol $\xi$, $\|\xi\|=\sum^d_{i=1}|\xi_i|$ and $\langle \xi \rangle = \sqrt{1+|\xi|^2}$.
According to the Paley-Wiener theorem\footnote{The proof given for $s=0$ in \cite{K} applies also for $s\in\mathbb{R}$
with some obvious modifications.} (see e.g. \cite{K}, p. 209),
a function $f$ belongs to $G^{\sigma,s}$ with $\sigma>0$
if and only if it is the restriction to the real line of a function $F$ which is holomorphic in the strip
$$S_{\sigma}=\{x+iy:x,y\in\mathbb{R}^d,|y_1|,|y_2|,\cdots,|y_d|<\sigma\}$$  and satisfies
$\sup_{|y| < \sigma} \| F(x+iy)\|_{H_x^s} <\infty$.
Therefore, every function in $G^{\sigma,s}$ with $\sigma>0$ has an analytic extension to the strip $S_\sigma$.

In view of this property of the Gevrey class, 
we take initial data in $G^{\sigma,s}$ for some initial radius $\sigma>0$
and then try to extend it globally in time with estimating the behavior of the radius of analyticity $\sigma(t)$ as time $t$ goes.
Our result is the following theorem.

\begin{thm}\label{maintheorem}
Let $d=1,2,3$. Let $(u,n)$ be the global $C^\infty$ solution of the Cauchy problem \eqref{KGS}
with the initial data $(u_0, n_0, n_1)\in G^{\sigma_0,r}(\mathbb{R}^d)\times G^{\sigma_0,s}(\mathbb{R}^d)\times G^{\sigma_0,s-1}(\mathbb{R}^d)$
for some $\sigma_0>0$ and $r,s\in\mathbb{R}$. Then for all $t\in\mathbb{R}$
$$
(u, n, \partial_t n)(t)\in G^{\sigma(t),r}(\mathbb{R}^d) \times G^{\sigma(t),s}(\mathbb{R}^d) \times G^{\sigma(t),s-1}(\mathbb{R}^d),
$$
where the radius of analyticity $\sigma(t)$ satisfies an asymptotic lower bound\,\footnote{We use the notation $a^\pm=a\pm\varepsilon$ for arbitrarily small $\varepsilon>0$.}
$$
\sigma(t)\geq ct^{-p^+}\quad \textrm{as}\quad |t|\rightarrow\infty,
$$
with $p=\max\{8/(4-d),4\}$ and a constant $c>0$ depending on $\sigma_0,r,s$ and the norm of the initial data.
\end{thm}

Only when $d=1,2,3$ does the existing well-posedness theory in $H^s$ guarantee the existence of the global $C^\infty$ solution
in the theorem, given initial data $(u_0, n_0, n_1)\in G^{\sigma_0,r}(\mathbb{R}^d)\times G^{\sigma_0,s}(\mathbb{R}^d)\times G^{\sigma_0,s-1}(\mathbb{R}^d)$
for any $\sigma_0>0$ and $r,s\in\mathbb{R}$.
Indeed, observe first that $G^{0,s}$ coincides with the Sobolev space $H^s$ and the embeddings
\begin{equation}\label{emb}
G^{\sigma,s}\subset G^{\sigma',s'}
\end{equation}
hold for all $0\leq \sigma' < \sigma$ and $s,s'\in\mathbb{R}$.
As a consequence of this embedding with $\sigma'=0$ and the existing well-posedness theory in $H^{s'}$,
the Cauchy problem \eqref{KGS} has a unique global smooth solution.

The outline of this paper is as follows. In Section \ref{sec2} we transform the system \eqref{KGS}
into an equivalent system of first order in time, and restate Theorem \ref{maintheorem} accordingly
(see Theorem \ref{maintheorem2}).
In Section \ref{sec3} we introduce some analytic function spaces and their basic properties
to be used in later sections.
In Section \ref{sec4} we present a couple of norm estimates employed in the proof of Theorem \ref{maintheorem2}. In Section \ref{sec5}, we first obtain a local-wellposedness in the Gevrey space applying Picard's iteration, and establish an almost conservation law to bound the growth of the Gevrey norm of the solution in time. Then we finish the proof by extending the local solution globally in time by making use of the approximate conservation law. The last section, Section \ref{sec6}, is devoted to proving the norm estimates given in Section \ref{sec4}.

Throughout this paper, we denote $A\lesssim B$ to mean $A\leq CB$ for some positive constant $C$, and $A\sim B$ to mean $A\lesssim B\lesssim A$.

\section{Reformulation of the system}\label{sec2}
We shall transform the system \eqref{KGS} into an equivalent system of first order in $t$ 
by observing 
$$\partial_t^2+1-\Delta=\langle D \rangle^2(1+i\langle D \rangle^{-1}\partial_t)(1-i\langle D \rangle^{-1}\partial_t).$$
We first let
$$n_{\pm} = n \pm i\langle D \rangle^{-1} \partial_t n.$$
Then we have
$$
n=\frac{1}{2}(n_+ + n_-),
$$
and the system becomes equivalent to 
\begin{equation}\label{KGS2}
\begin{cases}
i\partial_t u + \Delta u = -u (n_+ + n_-)/2, \quad\quad u(0) = u_0, \\
i\partial_t n_+ - \langle D \rangle n_+  = -\langle D \rangle^{-1} |u|^2, \quad\quad n_+(0) = \phi_+,\\
i\partial_t n_- + \langle D \rangle n_-  = \langle D \rangle^{-1} |u|^2, \quad\quad n_-(0) = \phi_-.
\end{cases}
\end{equation}
Notice that if
$$
(u,n,\partial_tn)\in C([-T,T]; G^{\sigma,r}(\mathbb{R}^d) \times G^{\sigma,s}(\mathbb{R}^d) \times G^{\sigma,s-1}(\mathbb{R}^d))
$$
is a solution of \eqref{KGS} with initial data $(u_0,n_0,n_1)\in G^{\sigma,r}\times G^{\sigma,s} \times G^{\sigma,s-1},$ then
$$
(u,n_+,n_-)\in C([-T,T]; G^{\sigma,r}(\mathbb{R}^d) \times G^{\sigma,s}(\mathbb{R}^d) \times G^{\sigma,s}(\mathbb{R}^d)
$$
is that of \eqref{KGS2} with initial data $(u_0,\phi_+,\phi_-)\in G^{\sigma,r}\times G^{\sigma,s} \times G^{\sigma,s},$ and vice versa.

With this observation we can restate Theorem \ref{maintheorem} as follows, and will prove the restatement in the remainder of the paper.

\begin{thm}[Theorem \ref{maintheorem}]\label{maintheorem2}
Let $d=1,2,3$.
Let $(u, n_+, n_-)$ be the global $C^\infty$ solution of the Cauchy problem \eqref{KGS2}
with the initial data
$(u_0, \phi_+, \phi_-)\in G^{\sigma_0,r}(\mathbb{R}^d)\times G^{\sigma_0,s}(\mathbb{R}^d)\times
 G^{\sigma_0,s}(\mathbb{R}^d)$
for some $\sigma_0>0$ and $r,s\in\mathbb{R}$. Then for all $t\in\mathbb{R}$
$$
(u, n_+, n_-)(t)\in G^{\sigma(t),r}(\mathbb{R}^d) \times G^{\sigma(t),s}(\mathbb{R}^d) \times G^{\sigma(t),s}(\mathbb{R}^d),
$$
where the radius of analyticity $\sigma(t)$ satisfies an asymptotic lower bound
$$
\sigma(t)\geq ct^{-p^+}\quad \textrm{as}\quad |t|\rightarrow\infty,
$$
with $p=\max\{8/(4-d),4\}$ and a constant $c>0$ depending on $\sigma_0,r,s$ and the norm of the initial data.
\end{thm}

\section{Preliminaries}\label{sec3}
In this section we introduce some function spaces and their basic properties which will be used later
for the proof of Theorem \ref{maintheorem2}.

For $s,b\in\mathbb{R}$ and some real valued polynomial $h$, we use  $X_h^{s,b}=X_h^{s,b}(\mathbb{R}^{1+d})$ to denote the Bourgain space defined by the norm
$$
\|f\|_{X^{s,b}_h} =\big\| \langle\xi\rangle^s\langle\tau-h(\xi)\rangle^b\widehat{f}(\tau,\xi)\big\|_{L^2_{\tau,\xi}},
$$
where $\widehat{f}$ denotes the space-time Fourier transform given by
$$
\widehat{f}(\tau,\xi)=\int_{\mathbb{R}^{1+d}} e^{-i(t\tau+x\cdot\xi)}f(t,x) \ dtdx.
$$
For simplicity, we omit $h$ in the notation $X^{s,b}_h$ when $h(\xi)=-|\xi|^2$, and replace $h$ with $\pm$ when $h(\xi)=\mp|\xi|$.
Since $\langle \tau\mp\langle \xi\rangle\rangle\sim\langle \tau\mp|\xi|\rangle$, 
we shall use here $\mp|\xi|$ just for technical reasons rather than $\mp\langle\xi \rangle$ for the Klein-Gordon evolution.
We denote by $X_h^{s,b}(\delta)$
the restriction of the Bourgain space to a time slab $(0,\delta )\times\mathbb{R}^d$
equipped with the norm
$$
\|f\|_{X^{s,b}_h(\delta)}=\inf\big\{\|g\|_{X^{s,b}_h} : g=f\,\, \text{on}\,\, (0,\delta)\times\mathbb{R}^d\big\}.
$$

We also introduce the Gevrey-Bourgain space $X_h^{\sigma,s,b}=X_h^{\sigma,s,b}(\mathbb{R}^{1+d})$ defined by the norm
$$
\|f\|_{X^{\sigma,s,b}_h}=\big\| e^{\sigma\| D\|}f\big\|_{X^{s,b}_h},
$$
which coincides with the Bourgain space $X^{s,b}_h$ particularly when $\sigma=0$.
Its restriction $X^{\sigma,s,b}_h(\delta)$ to a time slab $(0,\delta)\times\mathbb{R}^d$ is defined in a similar way as above.
The Gevrey-modification of the Bourgain spaces was used already by Bourgain \cite{B}
to study spatial analyticity for the Kadomtsev-Petviashvili equation.
He proved that the radius of analyticity remains positive as long as the solution exists.
His argument is quite general and applies to a class of dispersive equations,
but it does not give any lower bound on the radius $\sigma(t)$ as $|t|\rightarrow\infty$.

The $X_h^{\sigma,s,b}$-estimates in Lemmas \ref{lem1}, \ref{lem2} and \ref{lem3} follow easily by substitution $f\rightarrow e^{\sigma\|D\|}f$ using the corresponding properties of $X_h^{s,b}$-spaces and the restrictions thereof;
the proofs of the first two lemmas (when $\sigma=0$) can be found in Section 2.6 of \cite{T},
and see Lemma 7 of \cite{S} for the third lemma.

\begin{lem}\label{lem1}
Let $\sigma \geq 0$, $s\in \mathbb{R}$ and $b>1/2$. Then $X^{\sigma,s,b}_{h} \subset C(\mathbb{R},G^{\sigma,s})$ and
$$
\sup_{t\in\mathbb{R}} \| f(t)\|_{G^{\sigma,s}} \leq C\| f \|_{X^{\sigma,s,b}_{h}},
$$
where $C>0$ is a constant depending only on b.
\end{lem}

\begin{lem}\label{lem2}
Let $\sigma \geq 0$, $s\in\mathbb{R}$, $-1/2<b<b^\prime<1/2$ and $\delta>0$. Then
$$
\|f\|_{X^{\sigma,s,b}_h(\delta)} \leq C\delta^{b^\prime-b}\|f\|_{X^{\sigma,s,b^\prime}_h(\delta)},
$$
where the constant $C>0$ depends only on $b$ and $b'$.
\end{lem}

\begin{lem}\label{lem3}
Let $\sigma \geq 0$, $s\in\mathbb{R}$, $-1/2<b<1/2$ and $\delta>0$.
Then, for any time interval $I\subset[0,\delta]$,
$$
\|\chi_I f\|_{X^{\sigma,s,b}_h} \leq C\|f\|_{X^{\sigma,s,b}_h(\delta)},
$$
where $\chi_I(t)$ is the characteristic function of $I$, and the constant $C>0$ depends only on $b$.
\end{lem}

Finally we consider the Cauchy problem
$$
\begin{cases}(i\partial_t-h(D))u = F(t,x), \\
u(0,x)=f(x),
\end{cases}
$$
whose solution is written as
\begin{equation}\label{DF}
u(t,x)=e^{-ith(D)}f(x)-i\int_0^t e^{-i(t-s)h(D)}F(s,\cdot)ds
\end{equation}
with the Fourier multiplier $e^{-ith(D)}$ given by
$$e^{-ith(D)}f(x)=\frac1{(2\pi)^d}\int_{\mathbb{R}^d}e^{ix\cdot\xi}e^{-ith(\xi)}\widehat{f}(\xi)d\xi.$$
The following is nothing more than the standard energy estimate in $X_h^{s,b}(\delta)$-spaces (see Proposition 2.12 in \cite{T}):

\begin{lem}\label{sol_estimate}
Let $\sigma \geq 0$, $s\in\mathbb{R}$, $1/2<b\leq 1$ and $0<\delta\leq 1$. Then we have
$$
\big\| e^{-ith(D)}f\big\|_{X^{\sigma,s,b}_h(\delta)}\leq C\| f\|_{G^{\sigma,s}}
$$	
and
$$
\bigg\|\int^t_0e^{-i(t-s)h(D)}F(s,\cdot)ds\bigg\|_{X^{\sigma,s,b}_h(\delta)}\leq C \| F\|_{X^{\sigma,s,b-1}_{h}(\delta)}.
$$
Here the constant $C>0$ depends only on b.
\end{lem}

\section{Estimates in Gevrey-Bourgain spaces}\label{sec4}

Before proving Theorem \ref{maintheorem2} in earnest, we present in this section some estimates in Gevrey-Bourgain spaces.
The first ones are bilinear estimates in Lemma \ref{estimate} below. They will play a key role when estimating the product terms that appear in the system \eqref{KGS2} to ultimately obtain the desired local well-posedness in Gevrey spaces. We will see this first in the next secton, and prove the lemma thereafter.

\begin{lem}\label{estimate}
Let $d=1,2,3$. If $\sigma>0$ and $s>-1/2$, then we have
\begin{equation}\label{estimate1}
\|fg\|_{X^{\sigma,0,b'-1}}\lesssim\|f\|_{X^{\sigma,0,b}}\|g\|_{X^{\sigma,s,b}_\pm}
\end{equation}
and
\begin{equation}\label{estimate2}
\|f\bar{g}\|_{X_\pm^{\sigma,-s,b'-1}}\lesssim\|f\|_{X^{\sigma,0,b}}\|g\|_{X^{\sigma,0,b}}
\end{equation}
whenever $1/2<b\leq b'<\min\{(6+2s-d)/4,1,s+1\}$.
\end{lem}

\begin{rem}
As we will see later, uniform difference between $b$ and $b'$ when $s\leq0$ in the lemma
is needed to get a lower bound on the radius of spatial analyticity,
and larger difference yields better results. For $d=1$, it is known \cite{HGHL} that $b'-b$ can be (uniformly) as large as $(1/2)^-$, but for $d=2,3$, the largest possible difference known so far is essentially zero: see \cite{H,H2,H3,C}.
The lemma allows us to have significant differences for $d=2,3$ as well.
\end{rem}

When extending the local solution globally in time, we need to control the growth of the Gevrey norm of the solution.
We will carry out this by obtaining an approximate conservation law in the next section.
In this process the bilinear operator
\begin{equation}\label{F}
F(v,m):=vm- e^{\sigma\|D\|}(me^{-\sigma\|D\|}v)
\end{equation}
will appear and the following related estimates will play an important role:

\begin{lem}\label{F_lemma}
Let $d=1,2,3$ and $1/2<b\leq b'<\min\{(6-d)/4,1\}$. Then we have
\begin{equation*}
\|\overline{F(v,m)}\|_{X^{0,b'-1}}\lesssim\sigma\|v\|_{X^{0,b}}\|m\|_{X^{\sigma,1,b}_\pm}.
\end{equation*}
\end{lem}

\begin{proof}
We first take the space-time Fourier transform $\mathcal{F}$ of $\bar{F}$ to see
\begin{align*}
\mathcal{F}\big[\overline{F(v,m)}\big](\tau,\xi)
&=\overline{\int_{\mathbb{R}^{1+d}}\big(1-e^{\sigma(\|\xi\|-\|\xi-\xi_1\|)}\big)\widehat{v}(\tau_1-\tau,\xi_1-\xi)\widehat{m}(-\tau_1,-\xi_1) d\tau_1 d\xi_1}\\
&=\int_{\mathbb{R}^{1+d}}\big(1-e^{\sigma(\|\xi\|-\|\xi-\xi_1\|)}\big)\widehat{\bar{v}}(\tau-\tau_1,\xi-\xi_1)\widehat{\bar{m}}(\tau_1,\xi_1) d\tau_1 d \xi_1.
\end{align*}
Then we estimate
\begin{align*}
\big|\mathcal{F}\big[\overline{F(v,m)}\big](\tau,\xi)\big|
&\leq\int\big|\big(1-e^{\sigma(\|\xi\|-\|\xi-\xi_1\|)}\big)\widehat{\bar{v}}(\tau-\tau_1,\xi-\xi_1)\widehat{\bar{m}}(\tau_1,\xi_1)\big|d\tau_1 d \xi_1\\
&\leq \sigma\int|\widehat{\bar{v}}(\tau-\tau_1,\xi-\xi_1)|\|\xi_1\|e^{\sigma\| \xi_1\|}|\widehat{\bar{m}}(\tau_1,\xi_1)| d\tau_1 d \xi_1\\
&= \sigma\mathcal{F}\big[|v|e^{\sigma\|D\|}\partial_x|m|\big](\tau,\xi)
\end{align*}
using
$$
\big|e^{\sigma(\|\xi\|-\|\xi-\xi_1\|)}-1\big|\leq\big|e^{\sigma\| \xi_1\|}-1\big|
\leq \sigma\|\xi_1\|e^{\sigma\|\xi_1\|}
$$
which follows from the simple inequality
$e^x-1 \leq xe^x$ for $x\geq0$.
Finally, by \eqref{estimate1} with $s=0$, we obtain
\begin{align*}
\|\overline{F(v,m)}\|_{X^{0,b'-1}} &\leq\sigma\big\||v|e^{\sigma\|D\|}\partial_x |m|\big\|_{X^{0,b'-1}}\\
&\lesssim\sigma\big\||v|\big\|_{X^{0,b}}\big\|e^{\sigma\|D\|}\partial_x|m|\big\|_{X_\pm^{0,b}}\\
&\leq\sigma\|v\|_{X^{0,b}}\|m\|_{X^{\sigma,1,b}_\pm}.
\end{align*}
\end{proof}

\section{Proof of Theorem \ref{maintheorem2}}\label{sec5}
Now we are ready to prove Theorem \ref{maintheorem2}.
In Subsection \ref{sec5.1} we shall show that the radius of analyticity of the solution is positive locally in time, and then in Subsection \ref{sec5.2} we bound the growth of the Gevrey norm of the solution to apply the local result 
 repeatedly to cover time intervals of arbitrary length in Subsection \ref{sec5.3}.

\subsection{Local well-posedness}\label{sec5.1}
In view of Lemma \ref{lem1},
we basically apply Picard's iteration in the $X_h^{\sigma,s,b}(\delta)$-space to
establish the following local well-posedness in $G^{\sigma,0} \times G^{\sigma,1} \times G^{\sigma,1}$, with a lifespan $\delta > 0$.
Equally the radius of analyticity remains strictly positive in a short time interval $0 \leq t \leq\delta$,
where $\delta > 0$ depends on the norm of the initial data.

\begin{thm}\label{local_wellposedness2}
Let $d=1,2,3$ and $\sigma>0$.
Then, for any initial data $(u_0, \phi_+,\phi_-) \in G^{\sigma,0} \times G^{\sigma,1} \times G^{\sigma,1}$,
there exist $\delta>0$ and a unique\footnote{The uniqueness is immediate since the solution is certainly $C^\infty$.} solution
$$(u,n_+,n_-) \in C([0,\delta];G^{\sigma,0} \times G^{\sigma,1} \times G^{\sigma,1})$$
of the Cauchy problem \eqref{KGS2}.
Here we may take
\begin{equation}\label{delta}
\delta=C(1+\| u_0\|_{G^{\sigma,0}} +\| \phi_+\|_{G^{\sigma,1}} +\| \phi_-\|_{G^{\sigma,1}})^{-q^+}
\end{equation}
for $q=\max\{4/(4-d),2\}$ and for some constant $C>0$.
Furthermore, for $b=(1/2)^+$,
\begin{equation}\label{lowe1}
\| u\|_{X^{\sigma,0,b}(\delta)} \lesssim \| u_0\|_{G^{\sigma,0}}
\end{equation}
and
\begin{equation}\label{lowe2}
\| n_\pm \|_{X^{\sigma,1,b}_\pm (\delta)} \lesssim \|\phi_\pm\|_{G^{\sigma,1}} + \| u_0\|_{G^{\sigma,0}}.
\end{equation}

\end{thm}

\begin{proof}
Fix $\sigma>0$ and $(u_0, \phi_+, \phi_-) \in G^{\sigma,0} \times G^{\sigma,1} \times G^{\sigma,1}$.
By Lemma \ref{lem1} we shall employ an iteration argument in the space $X^{\sigma,0,b}(\delta)$ and $X^{\sigma,1,b}_{+}(\delta)$
instead of $G^{\sigma,0}$ and $G^{\sigma,1}$, respectively.
Define the Picard iterates $\big\{\big(u^{(k)},n_+^{(k)},n_-^{(k)}\big)\big\}^{\infty}_{k=0}$ by
$$
\begin{cases}
i\partial_t u^{(0)} + \Delta u^{(0)} = 0, &u^{(0)}(0,x) = u_0(x), \\
i\partial_t n^{(0)}_+ - \langle D \rangle n^{(0)}_+  = 0, &n^{(0)}_+(0,x) = \phi_+(x),\\
i\partial_t n^{(0)}_- + \langle D \rangle n^{(0)}_-  = 0, &n^{(0)}_-(0,x) = \phi_-(x),
\end{cases}
$$
and
$$
\begin{cases}
i\partial_t u^{(k)} + \Delta u^{(k)} = -u^{(k-1)} \big(n^{(k-1)}_+ + n^{(k-1)}_-\big)/2, &u^{(k)}(0,x) = u_0(x), \\
i\partial_t n^{(k)}_+ - \langle D \rangle n^{(k)}_+  = -\langle D \rangle^{-1} \big|u^{(k-1)}\big|^2, &n^{(k)}_+(0,x) = \phi_+(x),\\
i\partial_t n^{(k)}_- + \langle D \rangle n^{(k)}_-  = \langle D \rangle^{-1} \big|u^{(k-1)}\big|^2, &n^{(k)}_-(0,x) = \phi_-(x),
\end{cases}
$$
for $k\in \mathbb{Z}^+.$
By \eqref{DF}, we first write
\begin{equation}\label{localre1}
\begin{cases}
u^{(0)}(t,x) = e^{-it\Delta}u_0(x),\\
n_+^{(0)}(t,x) = e^{it\langle D\rangle}\phi_+(x),\\
n^{(0)}_-(t,x)=e^{-it\langle D\rangle}\phi_-(x),
\end{cases}
\end{equation}
and
\begin{equation}\label{localre2}
\begin{cases}
\displaystyle u^{(k)}(t,x)=e^{-it\Delta}u_0(x) - i\int_0^t e^{i(t-s)\Delta}u^{(k-1)}(s,\cdot)\frac{n^{(k-1)}_+ + n^{(k-1)}_-}{2}(s,\cdot) ds,\\
\displaystyle n^{(k)}_+(t,x) = e^{it\langle D\rangle}\phi_+(x) - i\int_0^t e^{-i(t-s)\langle D\rangle}( \langle D \rangle^{-1} |u^{(k-1)}(s,\cdot)|^2) ds,\\
\displaystyle n^{(k)}_-(t,x) = e^{-it\langle D\rangle}\phi_-(x) + i\int_0^t e^{i(t-s)\langle D\rangle}( \langle D \rangle^{-1} |u^{(k-1)}(s,\cdot)|^2) ds.
\end{cases}
\end{equation}

Now we show these sequences are Cauchy in Gevrey-Bourgain spaces. Applying Lemma \ref{sol_estimate} to \eqref{localre1} implies
\begin{equation}
\big\|u^{(0)} \big\|_{X^{\sigma,0,b}(\delta)}\leq C \|u_0\|_{G^{\sigma, 0}}\label{u0_estimate}\\
\end{equation}
and
\begin{equation}
\big\|n_\pm^{(0)} \big\|_{X^{\sigma,1,b}_\pm(\delta)}\leq C \|\phi_\pm\|_{G^{\sigma, 1}},\label{phi_estimate}
\end{equation}
while applying Lemmas \ref{sol_estimate} and \ref{lem2} to \eqref{localre2} yields
\begin{align}\label{uk_estimate_1}
 \nonumber\big\| u^{(k)}\big\|_{X^{\sigma,0,b}(\delta)}&\leq C \|u_0\|_{G^{\sigma,0}}
 + C\bigg\| u^{(k-1)}\frac{n^{(k-1)}_++n^{(k-1)}_-}{2}\bigg\|_{X^{\sigma,0,b-1}(\delta)}\\
&\leq C \|u_0\|_{G^{\sigma,0}} + C\delta^{b'-b}\bigg\| u^{(k-1)}\frac{n^{(k-1)}_++n^{(k-1)}_-}{2}\bigg\|_{X^{\sigma,0,b'-1}(\delta)}
\end{align}
and
\begin{equation}\label{nk_estimate_1}
\big\| n_\pm^{(k)}\big\|_{X_\pm^{\sigma,1,b}(\delta)}\leq C \|\phi_\pm\|_{G^{\sigma,1}} + C\delta^{b'-b}\big\||u^{(k-1)}|^2\big\|_{X^{\sigma,0,b'-1}_\pm(\delta)}
\end{equation}
with $1/2<b<b'<1$.
We then apply Lemma \ref{estimate} with $s=0$ to the last norms in \eqref{uk_estimate_1} and \eqref{nk_estimate_1}, respectively, to obtain
\begin{align}\label{uk_estimate}
\big\|u^{(k)}\big\|_{X^{\sigma,0,b}(\delta)}&\leq C \|u_0\|_{G^{\sigma,0}}\\
\nonumber&+C\delta^{b'-b}\big\|u^{(k-1)}\big\|_{X^{\sigma,0,b}(\delta)}
\big(\big\|n^{(k-1)}_+\big\|_{X^{\sigma,1,b}_+(\delta)}+\big\|n^{(k-1)}_-\big\|_{X^{\sigma,1,b}_-(\delta)}\big)
\end{align}
and
\begin{equation}\label{nk_estimate}
\big\|n^{(k)}_\pm \big\|_{X^{\sigma,1,b}_\pm(\delta)}
\leq C\| \phi_\pm\|_{G^{\sigma,1}} +C\delta^{b'-b}\big\| u^{(k-1)} \big\|^2_{X^{\sigma, 0,b}(\delta)}
\end{equation}
for $1/2<b<b'<\min\{(6-d)/4,1\}$. By induction together with \eqref{u0_estimate}, \eqref{phi_estimate}, \eqref{uk_estimate} and \eqref{nk_estimate},
it follows that for all $k\geq 0$
\begin{equation}\label{uk}
\big\| u^{(k)}\big\|_{X^{\sigma,0,b}(\delta)} \leq C \| u_0\|_{G^{\sigma,0}}
\end{equation}
and
\begin{equation}\label{nk}
\big\| n^{(k)}_\pm \big\|_{X^{\sigma,1,b}_\pm (\delta)} \leq C \|\phi_\pm\|_{G^{\sigma,1}} + C\| u_0\|_{G^{\sigma,0}}
\end{equation}
with a choice of $\delta$ as
\begin{equation}\label{delc}
\delta^{b'-b} = \frac 1{8C}\big(1+\| u_0\|_{G^{\sigma,0}} +\| \phi_+\|_{G^{\sigma,1}} +\| \phi_-\|_{G^{\sigma,1}}\big)^{-1}
\end{equation}
where $0<b'-b<\min\{(4-d)/4, 1/2\}$. Here we take $b'-b=(1/q)^-$ with $q$ as in \eqref{delta}. Similarly, we apply Lemmas \ref{sol_estimate} and \ref{lem2} with the same $\delta$ to yield
\begin{align*}
&\big\| u^{(k)}-u^{(k-1)}\big\|_{X^{\sigma,0,b}(\delta)}\\
&\leq C \delta^{(1/q)^-} \bigg\| u^{(k-1)}\frac{n^{(k-1)}_+-n^{(k-1)}_-}{2}-u^{(k-2)}\frac{n^{(k-2)}_+-n^{(k-2)}_-}{2}\bigg\|_{X^{\sigma,0,b'-1}(\delta)}\\
&=C\delta^{(1/q)^-} \bigg\| \big(u^{(k-1)}-u^{(k-2)}\big)\frac{n^{(k-1)}_+-n^{(k-1)}_-}{2}
+ \sum_{\pm} u^{(k-2)}\frac{n^{(k-1)}_\pm-n^{(k-2)}_\pm}{\pm2}\bigg\|_{X^{\sigma,0,b'-1}(\delta)}
\end{align*}
which can be in turn bounded by
\begin{align*}
C\delta^{(1/q)^-} &\sum_{\pm}\big(\|\phi_\pm\|_{G^{\sigma,1}}+\| u_0\|_{G^{\sigma,0}}\big) \big\|u^{(k-1)}-u^{(k-2)}\big\|_{X^{\sigma,0,b}(\delta)}\\
&+C\delta^{(1/q)^-} \sum_{\pm}\|u_0 \|_{G^{\sigma,0}}\big\|n^{(k-1)}_\pm-n^{(k-2)}_\pm\big\|_{X^{\sigma,1,b}_\pm(\delta)}
\end{align*}
using Lemma \ref{estimate} together with \eqref{uk} and \eqref{nk}.
Consequently, we get
\begin{equation*}
\big\| u^{(k)}-u^{(k-1)}\big\|_{X^{\sigma,0,b}(\delta)} \leq \frac{1}{4}\Big(\big\|u^{(k-1)}-u^{(k-2)}\big\|_{X^{\sigma,0,b}(\delta)} +\sum_\pm\big\|n^{(k-1)}_\pm-n^{(k-2)}_\pm\big\|_{X^{\sigma,1,b}_\pm(\delta)}\Big)
\end{equation*}
by \eqref{delc}.
Similarly, we have
\begin{equation*}
\big\| n^{(k)}_\pm - n^{(k-1)}_\pm \big\|_{X^{\sigma,1,b}_\pm(\delta)} \leq \frac{1}{8} \big\|u^{(k-1)}-u^{(k-2)}\big\|_{X^{\sigma,0,b}(\delta)}.
\end{equation*}
Combining these two estimates, we finally get
\begin{align*}
\big\| u^{(k)}-u^{(k-1)}\big\|&_{X^{\sigma,0,b}(\delta)}+\sum_\pm\big\| n^{(k)}_\pm - n^{(k-1)}_\pm \big\|_{X^{\sigma,1,b}_\pm(\delta)}\\
&\leq \frac{1}{2}\Big(\big\|u^{(k-1)}-u^{(k-2)}\big\|_{X^{\sigma,0,b}(\delta)} +\sum_\pm\big\| n^{(k-1)}_\pm - n^{(k-2)}_\pm\big\|_{X^{\sigma,1,b}_\pm(\delta)}\Big)
\end{align*}
which guarantees the convergence of the sequence $\big\{\big(u^{(k)},n_+^{(k)},n_-^{(k)}\big)\big\}^{\infty}_{k=0}$ to a solution $\{(u,n_+,n_-)\}$ with the bounds \eqref{uk} and \eqref{nk}. This also implies \eqref{lowe1} and \eqref{lowe2}.

\end{proof}

\subsection{Almost conservation law}\label{sec5.2}
Now we would like to apply repeatedly the local result just obtained above to cover time intervals of arbitrary length. This, of course, requires some sort of control on the growth of the Gevrey norm of the solution. This is because the local existence time depends on the norm at initial time. 

Let $\mathfrak{M}_{\sigma}(t)=\| u(t)\|^2_{G^{\sigma,0}}$ and $\mathfrak{N}_\sigma(t)=\|n_+(t)\|_{G^{\sigma,1}} + \| n_-(t)\|_{G^{\sigma,1}}$.
Then one can easily see
\begin{equation}\label{conservation}
\mathfrak{M}_{0}(t)\equiv\mathfrak{M}_{0}(0)
\end{equation}
which follows from the charge conservation of the system,
$$
\|u(t)\|_{L^2(\mathbb{R}^d)}\equiv\|u(0)\|_{L^2(\mathbb{R}^d)}.
$$
Although the conservation \eqref{conservation} fails to hold for $\sigma>0$,
the quantity $\mathfrak{M}_{\sigma}(t)$ is approximately conservative in the sense that the discrepancy between $\mathfrak{M}_{\sigma}(t)$ and $\mathfrak{M}_{\sigma}(0)$ is bounded appropriately. Similarly for $\mathfrak{N}_{\sigma}(t)$.
Indeed, the almost conservation in the following proposition will allow us in Subsection \ref{sec5.3} to repeat the local result on successive short time intervals to reach any target time $T>0$, by adjusting the strip width parameter $\sigma$ according to the size of $T$.
\begin{prop}\label{almost_conservation_law}
Let $d=1,2,3$, $\sigma>0$, $b=(1/2)^+$ and $\delta$ be as in Theorem \ref{local_wellposedness2}. 
For any solution $(u,n_+,n_-)\in X^{\sigma,0,b}(\delta)\times X^{\sigma,1,b}_+ (\delta)\times X^{\sigma,1,b}_-(\delta)$ to the Cauchy problem \eqref{KGS2} on the time interval $[0,\delta]$, we then have
\begin{equation}\label{U_estimate}
\sup_{t\in[0,\delta]} \mathfrak{M}_\sigma(t) \leq \mathfrak{M}_\sigma(0) + C\sigma\delta^{(1/q)^-}\mathfrak{M}_{\sigma}(0)\big(\mathfrak{M}^{1/2}_{\sigma}(0)+\mathfrak{N}_{\sigma}(0)\big)
\end{equation}
and
\begin{equation}\label{N_estimate}
\sup_{t\in[0,\delta]}\mathfrak{N}_{\sigma}(t) \leq\mathfrak{N}_{\sigma}(0) + C\delta^{(1/q)^-}\mathfrak{M}_\sigma(0)
\end{equation}
with the same $q$ as in Theorem \ref{local_wellposedness2} and a constant $C>0$.
\end{prop}

\begin{proof}
Let $0\leq \delta' \leq \delta$. Setting $v(t,x) = e^{\sigma\|D\|}u(t,x)$ and applying $e^{\sigma \|D\|}$ to the first equation in \eqref{KGS2}, we obtain
$$
(i\partial_t + \Delta) v = -v\bigg(\frac{n_++n_-}{2}\bigg) + F\bigg(v,\frac{n_++n_-}{2}\bigg)
$$
where $F$ is as in \eqref{F}. For brevity, here we shall simply write $F$ for $\displaystyle F\bigg(v,\frac{n_++n_-}{2}\bigg)$. Multiplying both sides by $\bar{v}$ and taking imaginary parts thereon, we see
$$
\text{Re}(\bar{v}\partial_tv) + \text{Im}(\bar{v}\Delta v) = \text{Im}( \bar{v}F),
$$
or equivalently
$$
\partial_t |v|^2 + 2\text{Im}(\bar{v}\Delta v) =2 \text{Im}( \bar{v}F)
$$
where we used the fact $\partial_t | v |^2=2\textrm{Re}(\bar{v}\partial_tv)$.
Integrating in space and using integration by parts, we also see
$$
\frac{d}{dt}\int_{\mathbb{R}^d} |v|^2 dx = 2\text{Im} \int_{\mathbb{R}^d} \bar{v}F dx,
$$
where we may assume that $v$ and its all spatial derivatives decay to zero as $|x| \rightarrow \infty$.\footnote{This property can be shown by approximation using the monotone convergence theorem
and the Riemann-Lebesgue lemma whenever $u\in X_h^{\sigma,1,b}(\delta)$. See the argument in \cite{SS}, p. 1018.}
Subsequently integrating in time over $[0,\delta']$, we now have
$$
\int_{\mathbb{R}^d} |v(\delta')|^2 dx = \int_{\mathbb{R}^d} |v(0)|^2 dx + 2\text{Im} \int_{\mathbb{R}^{1+d}} \chi_{[0,\delta']}(t)\bar{v}F dtdx.
$$
Using H\"{o}lder's inequality, and then Lemmas \ref{lem3} and \ref{lem2} as before, the rightmost integral is bounded as,
\begin{align*}
\bigg|\int_{\mathbb{R}^{1+d}} \chi_{[0,\delta']}(t)\bar{v}F dtdx\bigg|&\leq \| \chi_{[0,\delta']}v \|_{X^{0,1-b}}\|\chi_{[0,\delta']}\overline{F}\|_{X^{0,b-1}} \\
&\lesssim \delta^{(1/q)^-}\| u \|_{X^{\sigma, 0,1-b}(\delta')}\|\overline{F}\|_{X^{0,b'-1}(\delta')},
\end{align*}
which is in turn bounded by
$$
\delta^{(1/q)^-}\sigma\|u\|_{X^{\sigma,0,1-b}(\delta')}\|u\|_{X^{\sigma,0,b}(\delta')}(\|n_+\|_{X^{\sigma,1,b}_+(\delta')}+\|n_-\|_{X^{\sigma,1,b}_-(\delta')})
$$
from Lemma \ref{F_lemma}.
By applying \eqref{lowe1} and \eqref{lowe2} we therefore get
$$
\|u(\delta')\|^2_{G^{\sigma,0}} \leq \| u_0\|^2_{G^{\sigma,0}}+C\delta^{(1/q)^-}\sigma\| u\|^2_{G^{\sigma,0}}(\| u\|_{G^{\sigma,0}} + \|\phi_+\|_{G^{\sigma,1}}+\|\phi_-\|_{G^{\sigma,1}})
$$
which implies \eqref{U_estimate}.

Next we show \eqref{N_estimate}.
By \eqref{DF} we first note
$$n_\pm(t,x)=e^{\pm it\langle D \rangle} \phi_\pm(x) \mp i\int^t_0 e^{\mp i(t-s)\langle D\rangle}(\langle D\rangle^{-1}|u(s,\cdot)|^2)ds.$$
By Lemma \ref{lem1} we then see
\begin{equation*}
\| n_\pm(t)\|_{G^{\sigma,1}} \leq \|  \phi_\pm(x) \|_{G^{\sigma,1}} +C\|\langle D\rangle^{-1}|u|^2\|_{X^{\sigma,1,b-1}_\pm(\delta)},
\end{equation*}
while
\begin{align*}
\|\langle D\rangle^{-1}|u|^2\|_{X^{\sigma,1,b-1}_\pm(\delta)}
&\lesssim\delta^{(1/q)^-}\||u|^2\|_{X^{\sigma,0,b'-1}_\pm(\delta)}\\
&\lesssim\delta^{(1/q)^-}\|u\|^2_{X^{\sigma,0,b}(\delta)}\\
&\lesssim\delta^{(1/q)^-}\|u_0\|^2_{G^{\sigma,0}}
\end{align*}
by Lemma \ref{lem2}, \eqref{estimate2} and \eqref{lowe2}.
Summing up, we get for any $t\in[0,\delta]$
$$
\sum_{\pm}\| n_\pm(t)\|_{G^{\sigma,1}} \leq \sum_{\pm}\| \phi_\pm\|_{G^{\sigma,1}}+C\delta^{(1/q)^-}\|u_0\|^2_{G^{\sigma,0}}
$$
which implies \eqref{N_estimate}.

\end{proof}

\subsection{Global extension and radius of analyticity}\label{sec5.3}
Lastly we put it all together to complete the proof of Theorem \ref{maintheorem2}. By the embedding \eqref{emb}, the general case $(r,s)\in\mathbb{R}^2$ reduces to $(r,s)=(0,1)$ as shown at the end of this subsection.

Let us now prove the case $(r,s)=(0,1)$.
Given $\sigma_0>0$ and data such that $\mathfrak{M}_{\sigma_0}(0)$ and $\mathfrak{N}_{\sigma_0}(0)$ are finite, we must prove that for all large $T$, the solution has a positive radius of analyticity
\begin{equation}\label{m8m}
\sigma(t)\geq ct^{-p^+}\quad \textrm{for all}\ t\in[0,T],
\end{equation}
where $c>0$ is a constant depending on the data norms $\mathfrak{M}_{\sigma_0}(0)$ and $\mathfrak{N}_{\sigma_0}(0)$.
Now fix $T>1$ arbitrarily large and let $A\gg1$ denote a constant which may depend on $\mathfrak{M}_{\sigma_0}(0)$ and $\mathfrak{N}_{\sigma_0}(0)$; the choice of $A$ will be made explicit below. Let $q$ be as in Theorem \ref{local_wellposedness2}. It suffices to show that for all $t\in[0,T]$
\begin{equation}\label{label0}
\mathfrak{M}_{\sigma(t)}(t)\leq 4\mathfrak{M}_{\sigma(t)}(0)
\end{equation}
and
\begin{equation}\label{label00}
\mathfrak{N}_{\sigma(t)}(t)\leq 2AT^{q^+}
\end{equation}
with $\sigma(t)\leq\sigma_0$ satisfying \eqref{m8m}, which in turn implies $(u,n_+,n_-)(t)\in G^{\sigma(t),0}\times G^{\sigma(t),1}\times G^{\sigma(t),1}$ as desired. For brevity we denote $\sigma=\sigma(t)$.
To prove \eqref{label0} and \eqref{label00}, we may first assume that
\begin{equation}\label{aeonslater}
\mathfrak{M}^{1/2}_{\sigma_0}(0)+\mathfrak{N}_{\sigma_0}(0)\leq T^{q+\varepsilon_0}
\end{equation}
for some small $\varepsilon_0 >0$ since we are considering large $T$. Then we let
\begin{equation}\label{plancktime}
\delta=\frac{c_0}{AT^{q+\varepsilon_0}},
\end{equation}
for some small $c_0>0$ such that $n:=T/\delta$ is an integer.
It suffices to show for any $k\in\{1,2,\cdots,n\}$ and for some small $\varepsilon>0$ that
\begin{align}
\sup_{t\in[0,k\delta]}\mathfrak{M}_{\sigma}(t)&\leq\mathfrak{M}_{\sigma}(0)+kC\sigma\delta^{1/q-\varepsilon}(4\mathfrak{M}_{\sigma}(0))(4AT^{q+\varepsilon_0}),\label{label1}\\
\sup_{t\in[0,k\delta]}\mathfrak{N}_{\sigma}(t)&\leq\mathfrak{N}_{\sigma}(0)+kC\delta^{1/q-\varepsilon}(4\mathfrak{M}_\sigma(0)),\label{label11}
\end{align}
provided
\begin{align}
&nC\sigma\delta^{1/q-\varepsilon}(4\mathfrak{M}_{\sigma}(0))(4AT^{q+\varepsilon_0})\leq\mathfrak{M}_{\sigma}(0),\label{goodeq}\\
&nC\delta^{1/q-\varepsilon}(4\mathfrak{M}_\sigma(0))\leq AT^{q+\varepsilon_0}.\label{goodeq2}
\end{align}
For this, we shall use induction. The case $k=1$ is immediate from \eqref{U_estimate}, \eqref{N_estimate} and \eqref{aeonslater}. Now assume \eqref{label1} and \eqref{label11} hold for some $k\in\{1,2,\cdots,n-1\}$. By this assumption, \eqref{label0} and \eqref{label00} hold for $t\in[0,(n-1)\delta]$. Hence applying \eqref{U_estimate} with $k\delta$ as the initial time we have
\begin{align*}
\sup_{t\in[k\delta,(k+1)\delta]} \mathfrak{M}_\sigma(t)&\leq\mathfrak{M}_\sigma(k\delta)+C\sigma\delta^{1/q-\varepsilon}\mathfrak{M}_{\sigma}(k\delta)\big(\mathfrak{M}^{1/2}_{\sigma}(k\delta)+\mathfrak{N}_{\sigma}(k\delta)\big)\\
&\leq \mathfrak{M}_\sigma(k\delta)+C\sigma\delta^{1/q-\varepsilon}(4\mathfrak{M}_{\sigma}(0))\big(2\mathfrak{M}^{1/2}_{\sigma}(0)+2AT^{q+\varepsilon_0}\big),
\end{align*}
which is in turn bounded by
\begin{equation*}
\mathfrak{M}_\sigma(0)+(k+1)C\sigma\delta^{1/q-\varepsilon}\mathfrak{M}_{\sigma}(0)(4AT^{q+\varepsilon_0}),
\end{equation*}
using \eqref{aeonslater} and \eqref{label1}. In the same manner, we get
$$
\sup_{t\in[k\delta,(k+1)\delta]}\mathfrak{N}_{\sigma}(t)\leq\mathfrak{N}_{\sigma}(0)+(k+1)C\delta^{1/q-\varepsilon}(4\mathfrak{M}_\sigma(0)).
$$

We will show that \eqref{goodeq} and \eqref{goodeq2} hold under \eqref{m8m}. From $T=n\delta$ and \eqref{plancktime}, we get $n\delta^{1/q-\varepsilon}=T\delta^{(1-q)/q-\varepsilon}=c_1A^{(q-1)/q+\varepsilon}T^{1-(q+\varepsilon_0)((1-q)/q-\varepsilon)}$ where $c_1$ is an absolute constant. We can choose $\varepsilon>0$ such that $-(q+\varepsilon_0)((1-q)/q-\varepsilon)=q-1+\varepsilon_0$, which gives $n\delta^{1/q-\varepsilon}=c_1A^{(q-1)/q+\varepsilon}T^{q+\varepsilon_0}$; simple calculation shows that $\varepsilon\rightarrow0$ as $\varepsilon_0\rightarrow0$. Therefore \eqref{goodeq} and \eqref{goodeq2} reduce to
\begin{align}
&C\sigma c_1A^{(q-1)/q+\varepsilon}T^{q+\varepsilon_0}(16AT^{q+\varepsilon_0})\leq1,\label{finally}\\
&Cc_1A^{(q-1)/q+\varepsilon}T^{q+\varepsilon_0}(4\mathfrak{M}_\sigma(0))\leq AT^{q+\varepsilon_0}.\label{finally2}
\end{align}
To satisfy \eqref{finally2} we choose $A$ so large that
$$
Cc_1(4\mathfrak{M}_\sigma(0))\leq A^{1/q-\varepsilon}.
$$
Finally, \eqref{finally} is satisfied if $\sigma = cT^{-2q-2\varepsilon_0}$ where $c$ is a constant that may depend on $\mathfrak{M}_{\sigma_0}(0)$ and $\mathfrak{N}_{\sigma_0}(0)$. Since $2q=p$ and $\varepsilon_0$ can be arbitrarily small, we get the radius of analyticity \eqref{m8m}.

Now we consider the general case $(r,s)\in\mathbb{R}^2$. Recall that
$$
G^{\sigma,s}\subset G^{\sigma^\prime,s^\prime}\ \textrm{ for all }\ \sigma>\sigma'\geq 0\ \textrm{ and }\ s,s'\in\mathbb{R},
$$
from which we see that for any $(r,s)\in\mathbb{R}^2$,
\begin{align*}
(u_0,\phi_+,\phi_-) &\in G^{\sigma_0,r}(\mathbb{R}^d) \times G^{\sigma_0,s}(\mathbb{R}^d) \times G^{\sigma_0,s}(\mathbb{R}^d)\\
&\subset G^{\sigma_0/2,0}(\mathbb{R}^d) \times G^{\sigma_0/2,1}(\mathbb{R}^d) \times G^{\sigma_0/2,1}(\mathbb{R}^d).
\end{align*}
From the local theory there is a $\delta$ such that
$$
(u(t),n_+(t),n_-(t))\in G^{\sigma_0/2,0}(\mathbb{R}^d) \times G^{\sigma_0/2,1}(\mathbb{R}^d) \times G^{\sigma_0/2,1}(\mathbb{R}^d)\quad\textrm{for }\ 0\leq t\leq\delta\textrm{.}
$$
As in the case $(r,s)=(0,1)$, for fixed $T$ greater than $\delta$, we have $(u(t),n_+(t),n_-(t))\in G^{\sigma',0}(\mathbb{R}^d) \times G^{\sigma',1}(\mathbb{R}^d) \times G^{\sigma',1}(\mathbb{R}^d)$ for $t\in[0,T]$
and $\sigma^\prime\geq cT^{-p^+}$ with the constant $c>0$ depending on the data norms $\mathfrak{M}_{\sigma_0/2}(0)$ and $\mathfrak{N}_{\sigma_0/2}(0)$.
Applying the embedding again, we conclude
$$
(u(t),n_+(t),n_-(t))\in G^{\sigma,r}(\mathbb{R}^d) \times G^{\sigma,s}(\mathbb{R}^d) \times G^{\sigma,s}(\mathbb{R}^d)\quad\text{for}\,\ t\in [0,T]
$$
where $\sigma =\sigma^\prime/2$. 
This completes the proof.
\qed

\section{Proof of Lemma \ref{estimate}}\label{sec6}
This final section is devoted to the proof of Lemma \ref{estimate}.
Note first that
\begin{equation*}
\| fg\|_{X_h^{\sigma,s,b}} \leq \| (e^{\sigma\|D\|}f)(e^{\sigma\|D\|}g)\|_{X_h^{s,b}}
\end{equation*}
by the definitions of the norms. From this observation, \eqref{estimate1} and \eqref{estimate2} reduce to showing the case $\sigma=0$:
$$
\|fg\|_{X^{0,b'-1}}\lesssim\|f\|_{X^{0,b}}\|g\|_{X^{s,b}_\pm}
$$
and
$$
\|f\bar{g}\|_{X_\pm^{-s,b'-1}}\lesssim\|f\|_{X^{0,b}}\|g\|_{X^{0,b}}.
$$
The proof for $d=1$ can be found in \cite{HGHL}. Now let $d=2,3$. By the definition of $X^{s,b}$-norms and the dual characterisation of $L^2$ space, we may show that
\begin{gather}\label{part1}
\iiiint\limits_{\xi_0+\xi_1+\xi_2=0\atop \tau_0+\tau_1+\tau_2=0} \frac{\langle \xi_2 \rangle^{-s}f_0(\xi_0,\tau_0)f_1(\xi_1,\tau_1)f_2(\xi_2,\tau_2)}{\langle \tau_0-|\xi_0|^2\rangle^{1-b'}\langle \tau_1+|\xi_1|^2\rangle^{b}\langle \tau_2\pm|\xi_2|\rangle^{b}} d\xi_1d\xi_2d\tau_1d\tau_2 \lesssim\prod_{j=0}^2 \|f_j\|_{L^2_{\xi,\tau}}
\end{gather}
and
\begin{gather}\label{part2}
\iiiint\limits_{\xi_0+\xi_1+\xi_2=0\atop \tau_0+\tau_1+\tau_2=0}\frac{\langle \xi_2 \rangle^{-s}f_0(\xi_0,\tau_0)f_1(\xi_1,\tau_1)f_2(\xi_2,\tau_2)}{\langle \tau_0-|\xi_0|^2\rangle^{b}\langle \tau_1+|\xi_1|^2\rangle^{b}\langle \tau_2\pm|\xi_2|\rangle^{1-b'}} d\xi_1d\xi_2d\tau_1d\tau_2 \lesssim\prod_{j=0}^2 \|f_j\|_{L^2_{\xi,\tau}}.
\end{gather}
For this, the following lemma will be used repeatedly.
\begin{lem}[\cite{ET}]\label{ilemma}
If $\alpha>1$ and $\alpha\geq\beta\geq 0$, then
$$
\int_{\mathbb{R}}\frac{dy}{\langle y-a\rangle^{\alpha}\langle y-b\rangle^{\beta}}\lesssim \langle a-b\rangle^{-\beta}.
$$
\end{lem}
To show \eqref{part1} and \eqref{part2}, we first break the integration region into two parts: $|\xi_1|,|\xi_2|\lesssim 1$ (and thus $|\xi_0|\lesssim 1$) and the rest. The bound $|\xi_1|,|\xi_2|\lesssim 1$ will make the matter relatively simple. However, the second part requires a more delicate approach due to the absence of boundedness. Now we begin the first part.

\subsection{The case $|\xi_1|,|\xi_2|\lesssim 1$}\label{6.1}
In this case, it suffices to show
\begin{gather}\label{part1_case0}
\iiiint\limits_{\xi_0+\xi_1+\xi_2=0\atop \tau_0+\tau_1+\tau_2=0}\frac{f_0(\xi_0,\tau_0)f_1(\xi_1,\tau_1)f_2(\xi_2,\tau_2)}{\langle \tau_0-|\xi_0|^2\rangle^{1-b'}\langle \tau_1+|\xi_1|^2\rangle^{b}\langle \tau_2\pm|\xi_2|\rangle^{b}} d\xi_0d\xi_1d\tau_0d\tau_1 \lesssim\prod_{j=0}^2 \|f_j\|_{L^2_{\xi,\tau}}
\end{gather}
and
\begin{gather*}
\iiiint\limits_{\xi_0+\xi_1+\xi_2=0\atop \tau_0+\tau_1+\tau_2=0}\frac{f_0(\xi_0,\tau_0)f_1(\xi_1,\tau_1)f_2(\xi_2,\tau_2)}{\langle \tau_0-|\xi_0|^2\rangle^{b}\langle \tau_1+|\xi_1|^2\rangle^{b}\langle \tau_2\pm|\xi_2|\rangle^{1-b'}} d\xi_0d\xi_1d\tau_0d\tau_1 \lesssim\prod_{j=0}^2 \|f_j\|_{L^2_{\xi,\tau}}.
\end{gather*}
Since the latter can be obtained in a similar manner, we shall only show \eqref{part1_case0}. Using the H\"older inequality in $d\xi_0d\tau_0$ and then in $d\xi_1d\tau_1$, the left hand side of \eqref{part1_case0} is bounded as
\begin{align*}
&\|f_0\|_{L^2_{\xi_0,\tau_0}}\bigg\|  \iint\frac{f_1(\xi_1,\tau_1)f_2(-\xi_0-\xi_1,-\tau_0-\tau_1)}{\langle \tau_0-|\xi_0|^2\rangle^{1-b'}\langle \tau_1+|\xi_1|^2\rangle^{b}\langle \tau_0+\tau_1\pm|\xi_0+\xi_1|\rangle^{b}} d\xi_1d\tau_1\bigg\|_{L^2_{\xi_0,\tau_0}}\\
&\leq\|f_0\|_{L^2_{\xi_0,\tau_0}}\bigg\|  \iint f^2_1(\xi_1,\tau_1)f^2_2(-\xi_0-\xi_1,-\tau_0-\tau_1)d\xi_1d\tau_1\\
&\qquad\qquad\qquad\qquad\qquad\times\iint\frac{\langle \tau_0-|\xi_0|^2\rangle^{2(b'-1)}}{\langle \tau_1+|\xi_1|^2\rangle^{2b}\langle \tau_0+\tau_1\mp|\xi_0+\xi_1|\rangle^{2b}}d\xi_1d\tau_1 \bigg\|^{\frac{1}{2}}_{L^1_{\xi_0,\tau_0}}.
\end{align*}
Now the estimate \eqref{part1_case0} is obtained since
\begin{align*}
    \iint\frac{\langle \tau_0-|\xi_0|^2\rangle^{2b'-2}d\xi_1d\tau_1}{\langle \tau_1+|\xi_1|^2\rangle^{2b}\langle \tau_0+\tau_1\mp|\xi_0+\xi_1|\rangle^{2b}}&=\iint\frac{\langle \tau_0-|\xi_0|^2\rangle^{2b'-2}\langle \tau_1+|\xi_1|^2\rangle^{-2b}}{\langle \tau_0+\tau_1\mp|\xi_0+\xi_1|\rangle^{2b}}d\xi_1d\tau_1\\
    &\lesssim \iint\frac{\langle \tau_0+\tau_1-|\xi_0|^2+|\xi_1|^2\rangle^{2b'-2}}{\langle \tau_0+\tau_1\mp|\xi_0+\xi_1|\rangle^{2b}}d\xi_1d\tau_1\\
    &\lesssim \int\langle |\xi_0|^2-|\xi_1|^2\mp|\xi_0+\xi_1|\rangle^{2b'-2}d\xi_1<\infty
\end{align*}
uniformly in $\xi_0$ and $\tau_0$. This follows easily from using Lemma \ref{ilemma} together with $\langle a+b\rangle \lesssim \langle a\rangle\langle b\rangle$, $-2b<2b'-2<0$ and $|\xi_1|\lesssim 1$.
\qed

\subsection{The case $|\xi_1|\gg1$ or $|\xi_2|\gg1$}\label{6.2}
 Let $M_i$ be dyadic\footnote{This ranges over integer powers of 2.} numbers and $f_i^{M_i}=\chi_{\{|\xi_i|\sim M_i\}}f$
 so that $f_i=\sum_{M_i} f_i^{M_i}$.
 For simplicity, we drop the superscripts $M_i$ on $f_i^{M_i}$.
Then the left hand side of \eqref{part1} and \eqref{part2} are dyadically decomposed as
\begin{equation}\label{part1_decompose}
\sum_{M_0,M_1,M_2}\langle M_2 \rangle^{-s}\iiiint\limits_{\xi_0+\xi_1+\xi_2=0\atop \tau_0+\tau_1+\tau_2=0} \frac{f_0(\xi_0,\tau_0)f_1(\xi_1,\tau_1)f_2(\xi_2,\tau_2)}{\langle \tau_0-|\xi_0|^2\rangle^{1-b'}\langle \tau_1+|\xi_1|^2\rangle^{b}\langle \tau_2\pm|\xi_2|\rangle^{b}} d\xi_1d\xi_2d\tau_1d\tau_2
\end{equation}
and
\begin{equation}\label{part2_decompose}
\sum_{M_0,M_1,M_2}\langle M_2 \rangle^{-s}\iiiint\limits_{\xi_0+\xi_1+\xi_2=0\atop \tau_0+\tau_1+\tau_2=0} \frac{f_0(\xi_0,\tau_0)f_1(\xi_1,\tau_1)f_2(\xi_2,\tau_2)}{\langle \tau_0-|\xi_0|^2\rangle^{b}\langle \tau_1+|\xi_1|^2\rangle^{b}\langle \tau_2\pm|\xi_2|\rangle^{1-b'}}d\xi_1d\xi_2d\tau_1d\tau_2.
\end{equation}
Since $\tau_0+\tau_1+\tau_2=0$ and $\xi_0+\xi_1+\xi_2=0$, we also note
\begin{align*}
\max\{|\tau_0-|\xi_0|^2|,|\tau_1+|\xi_1|^2|,|\tau_2\pm|\xi_2||\}&\gtrsim \big|\tau_0-|\xi_0|^2+\tau_1+|\xi_1|^2+\tau_2\pm|\xi_2|\big|\\
&=\big||\xi_1+\xi_2|^2-|\xi_1|^2\mp|\xi_2|\big|\\
&=2|\xi_1||\xi_2||B|,
\end{align*}
where
\begin{equation}\label{BBB}
B =\cos(\alpha_{12})+\frac{|\xi_2|\mp1}{2|\xi_1|}
\end{equation}
with $\alpha_{12}$ being an angle between $\xi_1$ and $\xi_2$.
Now we set
\begin{equation}\label{M}
M=\max\{|\tau_0-|\xi_0|^2|,|\tau_1+|\xi_1|^2|,|\tau_2\pm|\xi_2||\}
\end{equation}
and then $M\gtrsim M_1M_2|B|$.
Let $h(\xi)=\pm|\xi|^2$ or $\pm|\xi|$. Then we may bound
\begin{equation*}
\sum_{M_0,M_1,M_2}\langle M_2 \rangle^{-s}\iiiint\limits_{\xi_0+\xi_1+\xi_2=0\atop \tau_0+\tau_1+\tau_2=0} \frac{f_0(\xi_0,\tau_0)f_1(\xi_1,\tau_1)f_2(\xi_2,\tau_2)}{\langle M\rangle^{1-b'}\langle \tau_{\sigma(0)}+h(\xi_{\sigma(0)})\rangle^{b}\langle \tau_{\sigma(1)}+h(\xi_{\sigma(1)})\rangle^{b}} d\xi_1d\xi_2d\tau_1d\tau_2
\end{equation*}
at one time for \eqref{part1_decompose} and \eqref{part2_decompose},
where $\sigma$ is a permuation in $\{0,1,2\}$. Here we used the fact that $\langle M\rangle^{1-b'} \langle\tau+h(\xi)\rangle^{b}\leq \langle M\rangle^{b} \langle\tau+h(\xi)\rangle^{1-b'}$ since $1/2<b\leq b'$.
To bound the above, we will first estimate the dyadic pieces in the sum and then check dyadic summability thereof.
Depending on which of the three becomes $M$ in \eqref{M}, we have three cases to examine.

We shall first consider the case when $M = |\tau_0-|\xi_0|^2|$.
We have to bound
\begin{equation}\label{serf}
\sum_{M_0,M_1,M_2}\langle M_2 \rangle^{-s}\iiiint\limits_{\xi_0+\xi_1+\xi_2=0\atop \tau_0+\tau_1+\tau_2=0} \frac{f_0(\xi_0,\tau_0)f_1(\xi_1,\tau_1)f_2(\xi_2,\tau_2)}{\langle M\rangle^{1-b'}\langle \tau_1+|\xi_1|^2\rangle^{b}\langle \tau_2\pm|\xi_2|\rangle^{b}} d\xi_1d\xi_2d\tau_1d\tau_2.
\end{equation}
Recall $M\gtrsim M_1M_2|B|$.
If $|B|\gtrsim 1$ then $M$ is bounded below by $M_1M_2$.
In this case we bound
\begin{equation}\label{case1_part1}
\sum_{M_0,M_1,M_2}\frac{\langle M_2 \rangle^{-s}}{ \langle M_1M_2\rangle^{1-b'}}\iiiint\limits_{\xi_0+\xi_1+\xi_2=0\atop \tau_0+\tau_1+\tau_2=0} \frac{f_0(\xi_0,\tau_0)f_1(\xi_1,\tau_1)f_2(\xi_2,\tau_2)}{\langle \tau_1+|\xi_1|^2\rangle^{b}\langle \tau_2\pm|\xi_2|\rangle^{b}} d\xi_1d\xi_2d\tau_1d\tau_2.
\end{equation}
However, for $|B|\ll 1$, the absence of such a bound again requires us further dyadic decomposition for $|B|$.
This finally leads to the following cases:
\begin{align*}
&a.\quad |B|\gtrsim 1,\quad  M_0 \lesssim M_1 \sim M_2,\\
&b.\quad |B|\gtrsim 1,\quad  M_2 \ll M_0 \sim M_1,\\
&c.\quad |B|\gtrsim 1,\quad  M_1 \ll M_0\sim M_2,\\
&d.\quad |B|\ll 1.
\end{align*}

\subsubsection*{Case a. $|B|\gtrsim 1, M_0 \lesssim M_1 \sim M_2$}
We decompose the functions $f_1$ and $f_2$ as
$$
f_1=\sum_{n\in\mathbb{Z}}f^n_1\quad \text{where}\quad f^n_1=\chi_{\{n-\frac{1}{2}\leq\tau_1+|\xi_1|^2< n+\frac{1}{2}\}}f_1
$$
and
$$
f_2=\sum_{n\in\mathbb{Z}}f^m_2\quad \text{where}\quad f^m_2=\chi_{\{m-\frac{1}{2}\leq\tau_2\pm|\xi_2|< m+\frac{1}{2}\}}f_2.
$$
Using these decompositions, the integral in \eqref{case1_part1} is bounded by
\begin{align}\label{case1_part11}
\nonumber \sum_{n\in\mathbb{Z}}\sum_{m\in \mathbb{Z}} &\langle n \rangle^{-b}\langle m \rangle^{-b} \iiiint\limits_{-\frac{1}{2}\leq\theta_i<\frac{1}{2}} f_0(-\xi_1-\xi_2, |\xi_1|^2\pm|\xi_2|-n-m-\theta_1-\theta_2)\\\
&\times f^n_1(\xi_1,-|\xi_1|^2+n+\theta_1)f^m_2(\xi_2,\mp|\xi_2|+m+\theta_2)d\xi_1d\xi_2d\theta_1d\theta_2
\end{align}
by the change of variables $\tau_1+|\xi_1|^2=n+\theta_1$ and $\tau_2 \pm |\xi_2|=m+\theta_2$.
To bound \eqref{case1_part11}, we first decompose
$$
f^n_{1,Q_i}=\chi_{\{\xi_1 \in Q_i\}}f^n_1\quad \text{and}\quad f^m_{2,R_j}=\chi_{\{\xi_2 \in R_j\}}f^m_2,
$$
where $Q_i$ and $R_j$ are essentially disjoint $d$-dimensional cubes of side length $M_0(\lesssim M_1,M_2)$ so that $\{|\xi_1|\sim M_1 \} =\cup_i Q_i$ and $\{|\xi_2|\sim M_2\} = \cup_j R_j$.
Then we consider the inner $d\xi_1d\xi_2$ integral in \eqref{case1_part11}; for fixed $\theta_i$, $n$ and $m$, first change variables $u=-\xi_1-\xi_2$ and $v=|\xi_1|^2\pm|\xi_2|-n-m-\theta_1-\theta_2$, replacing $\xi_1$ and one component of $\xi_2$, respectively. Let $\xi_i=(\xi_{i,1},\cdots,\xi_{i,d})$. Computing the determinant of the Jacobian matrix, we next see that
\begin{equation}\label{jaco}
dudvd\xi_{2,1}\cdots d\xi_{2,j-1}d\xi_{2,j+1}\cdots d\xi_{2,d}=\bigg|2\xi_{1,j}\pm\frac{\xi_{2,j}}{|\xi_2|}\bigg|d\xi_1d\xi_2.
\end{equation}
Since we may assume $M_1\gg 1$ in this case\footnote{We have already handled the case where all $M_i$ are small in the previous subsection.},
we have $|\xi_{1,j}|\sim M_1 \gg 1$ for some $j$.
Hence, for fixed $j$, the determinant of the Jacobian is nonzero in the region where $|\xi_{1,j}|\sim M_1 \gg 1$.
For this reason we divide the integration region in \eqref{case1_part11} into $d$ parts,
$|\xi_{1,j}|\sim M_1$ for $j=1,...,d$.
Without loss of generality, we may assume $j=1$. The inner integral is then rephrased as
$$
\sum_{Q_i}\iint\limits_{(\xi_{2,2},\cdots,\xi_{2,d})\atop\in\pi(R_{j(i)})} f_0(u,v)H_{Q_i}(u,v,\xi_{2,2},\cdots,\xi_{2,d})\bigg|2\xi_{1,1}\pm\frac{\xi_{2,1}}{|\xi_2|}\bigg|^{-1}dudvd\xi_{2,2}\cdots d\xi_{2,d},
$$
where $\pi : \mathbb{R}^d \rightarrow \mathbb{R}^{d-1}$ is the projection onto the last $d-1$ components and
$$
H_{Q_i}(u,v,\xi_{2,2},\cdots,\xi_{2,d}) =  f^n_{1,Q_i}(\xi_1,-|\xi_1|^2+n+\theta_1)f^m_{2,R_{j(i)}}(\xi_2,\mp|\xi_2|+m+\theta_2).
$$
Here we used the fact that the cube $R_j=R_{j(i)}$ is essentially determined by the cube $Q_i$
since $\xi_0+\xi_1+\xi_2=0$ and $M_0 \lesssim M_1\sim M_2$.
More precisely, each region $Q_i$ could correspond to up to at most $3^d$ of the $R_j$ regions.

By using H\"{o}lder's inequality twice, the above is bounded by
\begin{align*}
&M^{-1}_1 \|f_0\|_{L^2_{u,v}}\sum_{Q_i}\bigg\|\int_{(\xi_{2,2},\cdots,\xi_{2,d})\atop\in\pi(R_{j(i)})} H_{Q_i}(u,v,\xi_{2,2},\cdots,\xi_{2,d})\,d\xi_{2,2}\cdots \,d\xi_{2,d}\bigg\|_{L^2_{u,v}}\\
&\lesssim M^{-1}_1 \|f_0\|_{L^2_{u,v}}M^{(d-1)/2}_0\sum_{Q_i}\| H_{Q_i}(u,v,\xi_{2,2},\cdots,\xi_{2,d})\|_{L^2_{u,v,\xi_{2,2}\cdots \xi_{2,d}}}.
\end{align*}
Changing variables back to $(\xi_1,\xi_2)$ in the $L^2_{u,v,\xi_{2,2},\cdots,\xi_{2,d}}$ norm, we gain a factor of $M_1^{1/2}$ and observe that
\begin{align*}
&\sum_{Q_i}\| H_{Q_i}(u,v,\xi_{2,2},\cdots,\xi_{2,d})\|_{L^2_{u,v,\xi_{2,2}\cdots \xi_{2,d}}}\\
&=M^{1/2}_1 \sum_{Q_i}\| f^n_{1,Q_i}(\xi_1,-|\xi_1|^2+n+\theta_1)\|_{L^2_{\xi_1}}\|f^m_{2,R_{j(i)}}(\xi_2,\mp|\xi_2|+m+\theta_2)\|_{L^2_{\xi_2}}\\
&\lesssim M^{1/2}_1 \| f^n_1(\xi_1,-|\xi_1|^2+n+\theta_1)\|_{L^2_{\xi_1}}\|f^m_{2}(\xi_2,\mp|\xi_2|+m+\theta_2)\|_{L^2_{\xi_2}}.
\end{align*}
Here we used the Cauchy-Schwarz inequality in $Q_i$ and the disjointness of $Q_i$ for the last inequality.
Thus \eqref{case1_part11}, as a whole, is bounded by
\begin{align*}
&M^{(d-1)/2}_0M^{-1/2}_1 \|f_0\|_{L^2_{\xi,\tau}}\sum_{n\in\mathbb{Z}}\sum_{m\in \mathbb{Z}} \langle n \rangle^{-b}\langle m \rangle^{-b}&\\
&\qquad \times\iint\limits_{-\frac{1}{2}\leq \theta_i<\frac{1}{2}}\| f^n_1(\xi_1,-|\xi_1|^2+n+\theta_1)\|_{L^2_{\xi_1}}\|f^m_{2}(\xi_2,\mp|\xi_2|+m+\theta_2)\|_{L^2_{\xi_2}}d\theta_1d\theta_2.
\end{align*}
Using the Cauchy-Schwarz inequality in $\theta_1$ and $\theta_2$, and then in $n$ and $m$ with the fact that $b>1/2$, the double sum on $n,m$ here is bounded by
\begin{align*}
&\sum_{n\in\mathbb{Z}} \langle n \rangle^{-b}\| f^n_1(\xi_1,-|\xi_1|^2+n+\theta_1)\|_{L^2_{\xi_1,\theta_1}(-\frac{1}{2}\leq\theta_1<\frac{1}{2})}\\
&\quad\times\sum_{m\in \mathbb{Z}} \langle m \rangle^{-b}\|f^m_{2}(\xi_2,\mp|\xi_2|+m+\theta_2)\|_{L^2_{\xi_2,\theta_2}(-\frac{1}{2}\leq\theta_2<\frac{1}{2})}\\
&\leq \bigg(\sum_{n\in\mathbb{Z}} \langle n \rangle^{-2b}\bigg)^{\frac{1}{2}} \bigg(\sum_{n\in\mathbb{Z}}\| f^n_1(\xi_1,-|\xi_1|^2+n+\theta_1)\|^2_{L^2_{\xi_1,\theta_1}(-\frac{1}{2}\leq\theta_1<\frac{1}{2})}\bigg)^{\frac{1}{2}}\\
&\quad\times \bigg(\sum_{m\in\mathbb{Z}} \langle m \rangle^{-2b}\bigg)^{\frac{1}{2}} \bigg(\sum_{n\in\mathbb{Z}}\| f^m_2(\xi_2,\mp|\xi_2|+m+\theta_2)\|^2_{L^2_{\xi_2,\theta_2}(-\frac{1}{2}\leq\theta_2<\frac{1}{2})}\bigg)^{\frac{1}{2}}\\
&\lesssim \|f_1\|_{L^{2}_{\xi,\tau}}\|f_2\|_{L_{\xi,\tau}^2}.
\end{align*}
Therefore, \eqref{case1_part1} for \textit{Case a} is bounded by
$$
\sum_{M_0\lesssim M_1\sim M_2} \langle M_2\rangle^{-s}\langle M_1M_2\rangle^{b'-1}M_0^{\frac{d-1}{2}}M_1^{-\frac{1}{2}}\prod_{i=0}^2\|f_i\|_{L^2_{\xi,\tau}}.
$$
Since $M_1\gg1$ and $b'<1$, the sum here is bounded as
\begin{align*}
\sum_{M_0}M_0^{\frac{d-1}{2}}\sum_{M_1\gtrsim M_0} \langle M_1\rangle^{-s}\langle M_1^2\rangle^{b'-1}M_1^{-\frac{1}{2}}
&\lesssim
\sum_{M_0\gg1}M_0^{\frac{d-1}{2}}\sum_{M_1\gtrsim M_0} M_1^{-s+2(b'-1)-\frac{1}{2}}\\
&+\sum_{M_0\lesssim1}M_0^{\frac{d-1}{2}}\sum_{M_1\gg1} M_1^{-s-\frac{1}{2}},
\end{align*}
which is in turn bounded by
$$
\sum_{M_0\gg1}  M_0^{2b'-\frac{(2s+6)-d}{2}}+\sum_{M_0\lesssim1} M_0^{\frac{d-1}{2}}
$$
since $s>-1/2$.
This is finally summable under the assumption $b'<(2s+6-d)/4$.
\qed

\subsubsection*{Case b. $|B|\gtrsim 1,  M_2 \ll M_0 \sim M_1$}
Repeating the change of variables and the ensuing procedure as in \textit{Case a}, one can bound
\begin{equation*}
\iiiint\limits_{\xi_0+\xi_1+\xi_2=0\atop \tau_0+\tau_1+\tau_2=0} \frac{f_0(\xi_0,\tau_0)f_1(\xi_1,\tau_1)f_2(\xi_2,\tau_2)}{\langle \tau_1+|\xi_1|^2\rangle^{b}\langle \tau_2\pm|\xi_2|\rangle^{b}} d\xi_1d\xi_2d\tau_1d\tau_2
\lesssim
M_2^{\frac{d-1}{2}}M_1^{-\frac{1}{2}}\prod_{i=0}^2\|f_i\|_{L^2_{\xi,\tau}}
\end{equation*}
since we may assume $M_1\gg1$ as well in this case. The restriction $M_1\gg1$ is necessary to ensure that the Jacobian is nonzero.
Furthermore, no decomposition of the integration regions into cubes is required here.
The decomposition makes the projection on the integration region in $\xi_2$ onto any axis have measure at most $\min\{M_0,M_1,M_2\}$,
but it is automatically true when $M_2=\min\{M_0,M_1,M_2\}$.
Now, \eqref{case1_part1} for \textit{Case b} is bounded by
$$
\sum_{M_2\ll M_0\sim M_1} \langle M_2\rangle^{-s}\langle M_1M_2\rangle^{b'-1}M_2^{\frac{d-1}{2}}M_1^{-\frac{1}{2}}\prod_{i=0}^2\|f_i\|_{L^2_{\xi,\tau}}
$$
where the sum is finite as before;
\begin{align*}
\sum_{M_2}\langle M_2\rangle^{-s}M_2^{\frac{d-1}{2}}
\sum_{M_1\gg M_2} \langle M_1M_2\rangle^{b'-1}M_1^{-\frac{1}{2}}
&\lesssim\sum_{M_2\gg1} M_2^{-s+\frac{d-1}{2}+2(b'-1)}
\sum_{M_1\gg M_2}M_1^{-\frac{1}{2}}\\
&+\sum_{M_2\lesssim1}\langle M_2\rangle^{-s+2(b'-1)}M_2^{\frac{d-1}{2}}
\sum_{M_1\gg 1}M_1^{-\frac{1}{2}}\\
&\lesssim\sum_{M_2\gg1}M_2^{2b'-\frac{(2s+6)-d}{2}}+\sum_{M_2\lesssim1}M_2^{\frac{d-1}{2}}<\infty.
\end{align*}
\qed

\subsubsection*{Case c. $|B|\gtrsim 1, M_1 \ll M_0\sim M_2$}
For the same reason we can proceed as in \textit{Case a} when $M_1\gg1$. So we shall only consider the case $M_1\lesssim 1$.
 Since we cannot guarantee that the Jacobian is nonzero any more, we shall use a similar argument as in Section \ref{6.1}.
 Using H\"{o}lder's inequality, the integral in \eqref{case1_part1} is bounded as
\begin{align*}
&\iiiint\limits_{\xi_0+\xi_1+\xi_2=0\atop \tau_0+\tau_1+\tau_2=0} \frac{f_0(\xi_0,\tau_0)f_1(\xi_1,\tau_1)f_2(\xi_2,\tau_2)}{\langle \tau_1+|\xi_1|^2\rangle^b \langle\tau_2\pm|\xi_2|\rangle^{b}}d\xi_1d\xi_2d\tau_1d\tau_2\\
&\leq \|f_0\|_{L^{2}_{\xi_0,\tau_0}}\bigg\|\iint \frac{f_1(\xi_1,\tau_1)f_2(-\xi_0-\xi_1,-\tau_0-\tau_1)}{\langle\tau_1+|\xi_1|^2\rangle^b \langle\tau_0+\tau_1\mp|\xi_0+\xi_1|\rangle^b}d\xi_1d\tau_1\bigg\|_{L^{2}_{\xi_0,\tau_0}}\\
&\leq\|f_0\|_{L^{2}_{\xi_0,\tau_0}}\times\\
&\bigg\|  \iint f^2_1(\xi_1,\tau_1)f^2_2(-\xi_0-\xi_1,-\tau_0-\tau_1)d\xi_1d\tau_1\iint\limits_{|\xi_1|\sim M_1}\frac{\langle \tau_1-|\xi_1|^2\rangle^{-2b}d\xi_1d\tau_1}{\langle \tau_0+\tau_1\mp|\xi_0+\xi_1|\rangle^{2b}} \bigg\|^{\frac{1}{2}}_{L^{1}_{\xi_0,\tau_0}}.
\end{align*}
Here,
\begin{align*}
\sup_{\xi_0,\tau_0}\iint_{|\xi_1|\sim M_1}\frac{\langle \tau_1-|\xi_1|^2\rangle^{-2b}d\xi_1d\tau_1}{\langle \tau_0+\tau_1\mp|\xi_0+\xi_1|\rangle^{2b}}
&\lesssim \sup_{\xi_0,\tau_0}\iint_{|\xi_1|\sim M_1}\langle\tau_0+|\xi_1|^2\mp|\xi_0+\xi_1|\rangle^{-2b}d\xi_1\\
&\lesssim M_1^d
\end{align*}
using Lemma \ref{ilemma} with $\alpha=\beta=2b$. Hence \eqref{case1_part1} is bounded by
$$
\sum_{M_1\ll M_0\sim M_2,\,M_1\lesssim1}\langle M_2\rangle^{-s}\langle M_1M_2\rangle^{b'-1}M_1^{\frac{d}{2}}\prod_{i=0}^2\|f_i\|_{L^2_{\xi,\tau}}.
$$
Using the simple inequality $\langle ab \rangle\gtrsim |a|\langle b \rangle$ for $a\lesssim1$,
we bound the sum as
$$
\sum_{M_1\lesssim1}M_1^{\frac{d}{2}}\sum_{M_2\gg M_1}\langle M_2\rangle^{-s}\langle M_1M_2\rangle^{b'-1}
\lesssim\sum_{M_1\lesssim1}M_1^{\frac{d}{2}+b'-1}\sum_{M_2\gg 1}\langle M_2\rangle^{-s+b'-1}
$$
since we may assume $M_2\gg1$ (otherwise all $|\xi_j|\lesssim1$, and this case was already addressed in Subsection \ref{6.1}).
This is summable since $s>b'-1$ and $d\geq2>2-2b'$.
\qed

\begin{figure} [t]
 \begin{center}
  {\includegraphics[width=0.4\textwidth]{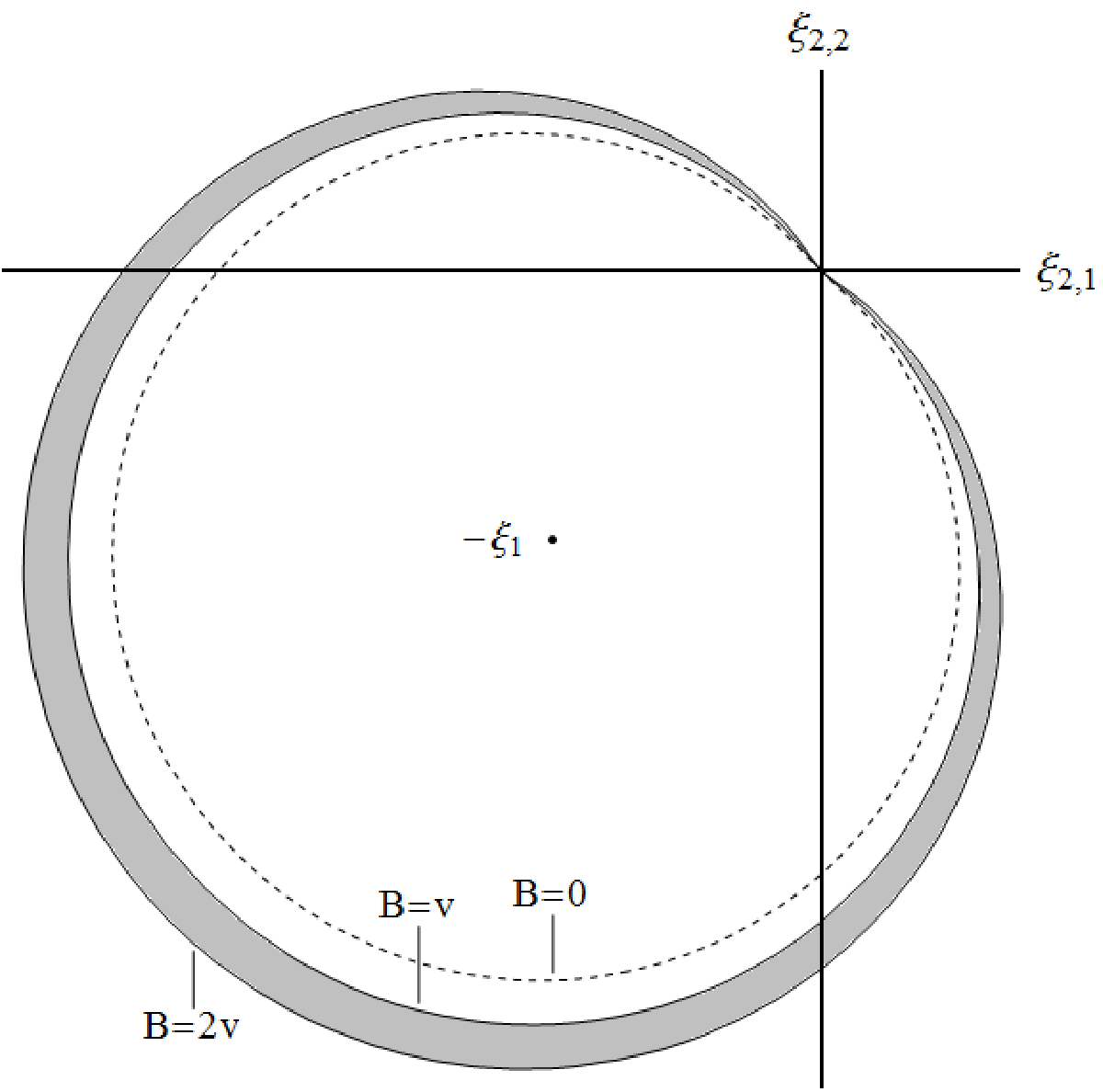}}
 \end{center}
 \caption{\label{figure}Integration region in $\xi_2$ for $d=2$}
\end{figure}

\subsubsection*{Case d. $|B|\ll 1$}

Recall first from \eqref{BBB} that
$$B=\cos{\alpha_{12}}+\frac{|\xi_2|\mp 1}{2|\xi_1|}.$$
Since  $|B|\ll 1$, we then see $|\xi_2|\lesssim |\xi_1|\pm1$, and hence $M_2\lesssim M_0\sim M_1$ in this case\footnote{We have either $M_2\ll M_1$ or $M_2\sim M_1$. The largest two of $M_0$, $M_1$, $M_2$ are comparable since $\xi_0+\xi_1+\xi_2=0$, and thus we have $M_1\not\ll M_0$ which, for the case $M_2\ll M_1$, clearly gives $M_0\sim M_1$. For the case $M_2\sim M_1$, if $M_0\ll M_1\sim M_2$ then $\xi_1$ is almost identical in size to $\xi_2$ but in the opposite direction. This implies $\cos(\alpha_{12})\approx-1$ and $\frac{|\xi_2|\mp 1}{2|\xi_1|}\approx1/2$ yet this contradicts that $|B|\ll 1$, which again gives $M_0\sim M_1$.}.
We may also assume that $|\xi_1|\gg1$. Otherwise all $|\xi_j|\lesssim1$, and this case was already addressed in Subsection \ref{6.1}.
Recalling $M\gtrsim M_1M_2|B|$ and decomposing dyadically $|B|\sim v\ll 1$, we need to bound
\begin{equation*}
\sum_{M_2\lesssim M_0\sim M_1}\frac{\langle M_2 \rangle^{-s}}{(M_1M_2)^{1-b'}}\sum_{v\ll 1} \frac{1}{v^{1-b'}}
\iiiint\limits_{\xi_0+\xi_1+\xi_2=0\atop \tau_0+\tau_1+\tau_2=0} \frac{f_0(\xi_0,\tau_0)f_1(\xi_1,\tau_1)f_2(\xi_2,\tau_2)}{\langle \tau_1+|\xi_1|^2\rangle^{b}\langle \tau_2\pm|\xi_2|\rangle^{b}} d\xi_1d\xi_2d\tau_1d\tau_2
\end{equation*}
this time instead of \eqref{case1_part1}.
The integral here boils down to
\begin{align}\label{case1_part14}
\nonumber\sum_{n\in\mathbb{Z}}&\sum_{m\in \mathbb{Z}} \langle n \rangle^{-b}\langle m \rangle^{-b} \iiiint\limits_{-\frac{1}{2}\leq\theta_i<\frac{1}{2}\atop |B|\sim v} f_0(-\xi_1-\xi_2, |\xi_1|^2\pm|\xi_2|-n-m-\theta_1-\theta_2)\\\
&\times f^n_1(\xi_1,-|\xi_1|^2+n+\theta_1)f^m_2(\xi_2,\mp|\xi_2|+m+\theta_2)d\xi_1d\xi_2d\theta_1d\theta_2
\end{align}
similarly as in \textit{Case a} (see \eqref{case1_part11}).
For a fixed $\xi_1$,
the curve $B=0$ satisfies $|\xi_2|^2+2|\xi_1||\xi_2|\cos\alpha_{12}\mp|\xi_2|=0$, which is a slightly distorted version of
a circle of radius $|\xi_1|$ centered at $-\xi_1$ given by equation $|\xi_2|^2+2|\xi_1||\xi_2|\cos\alpha_{12}=0$.
The region of integration in $\xi_2$ is given by a shell centered on the curve $B=0$, with thickness $\lesssim vM_1$ (see Figure \ref{figure}).
This holds since for a fixed $\xi_1$ and a fixed angle $\alpha_{12}$,
$$
|\xi_2| \in \big[2|\xi_1|(v-\cos{\alpha_{12}})\pm 1 , 2|\xi_1|(2v-\cos{\alpha_{12}})\pm1)\big]
$$
which is an interval of length $2v|\xi_1|$. This follows from $|B|\in [v,2v]$.

Now we decompose the annulus $\{\xi_1 : |\xi_1|\sim M_1 \}$ into two parts,
a set $S$ where $|\xi_{1,j}|\sim M_1$ for each $j$, and its complement.
In two dimensions, this can be described explicitly by taking
$$
S=\bigg\{\xi_1 : |\xi_1|\sim M_1, \arg{(\xi_1)}\in \bigg[\frac{\pi}{8},\frac{3\pi}{8}\bigg) \cup \bigg[\frac{5\pi}{8},\frac{7\pi}{8}\bigg) \cup \bigg[\frac{9\pi}{8},\frac{11\pi}{8}\bigg) \cup\bigg[\frac{13\pi}{8},\frac{15\pi}{8}\bigg) \bigg\}.
$$
Notice that the complement of $S$ is simply a $\pi/4$ radian rotation of $S$ about the origin.
In higher dimensions, the set $S$ is similarly given; if we describe the space in hyperspherical coordinates, we require all $d-1$ angular variables to be bounded away from multiples of $\pi/2$--specifically to fall within the intervals $[n\pi/8, (n+2)\pi/8)$ given above.
The complement of $S$ then consists of $2^{d-1}-1$ copies of $S$, each of which can be obtained from $S$ by a sequence of $\pi/4$ radian rotations.

Now we shall discuss the two-dimensional case in detail since there is no fundamental difference in higher dimensions.
We perform a rotation so that \eqref{case1_part14} can be written as a sum of two integrals over $S$ as
\begin{align*}
\nonumber\sum_{k=0}^1 &\sum_{n\in\mathbb{Z}}\sum_{m\in \mathbb{Z}} \langle n \rangle^{-b}\langle m \rangle^{-b}
\iiiint\limits_{-\frac{1}{2}\leq\theta_i<\frac{1}{2}\atop |B|\sim v, \xi_1\in S} f_0(R_{\frac{k\pi}{4}}(-\xi_1-\xi_2), |\xi_1|^2\pm|\xi_2|-n-m-\theta_1-\theta_2)\\
&\times f^n_1(R_{\frac{k\pi}{4}}(\xi_1),-|\xi_1|^2+n+\theta_1)f^m_2(R_{\frac{k\pi}{4}}(\xi_2),\mp|\xi_2|+m+\theta_2)d\xi_1d\xi_2d\theta_1d\theta_2,
\end{align*}
where $R_y$ denotes a rotation by $y$ radians.
Here we break the $d\xi_1d\xi_2$ integration into two regions:
one where for fixed $\xi_1$ and $\xi_{2,1}$, the projection of the integration region onto the $\xi_{2,2}$ axis is length $\lesssim vM_1$, and one where for fixed $\xi_1$ and $\xi_{2,2}$, the projection onto the $\xi_{2,1}$ axis is length $\lesssim vM_1.$
We then use the change of variables $u=-\xi_1-\xi_2$ and $v=|\xi_1|^2\pm|\xi_2|-n-m-\theta_1-\theta_2$ once again from \textit{Case a}.
In the first region, changing variables to replace $d\xi_1d\xi_2$ with $dudvd\xi_{2,2}$, we have
$$
d\xi_1d\xi_2=\bigg|2\xi_{1,1}\pm\frac{\xi_{2,1}}{|\xi_2|}\bigg|^{-1}dudvd\xi_{2,2}
\sim|\xi_{1,1}|^{-1}dudvd\xi_{2,2}
$$
where $|\xi_{1,1}|\sim M_1 \gg 1$ which guarantees the nonzero determinant of the Jacobian.
Similarly in the other region, we replace $d\xi_1d\xi_2$ with $dudvd\xi_{2,1}$.

Following exactly the same lines as in \textit{Case a} for each $k=0,1$, one can bound the above sum by
$$
\frac{(vM_1)^{\frac{1}{2}}}{M_1^{\frac{1}{2}}}\prod_{i=0}^{2}\|f_i\|_{L^{2}_{\xi,\tau}}.
$$
But here, we note that
$$
\sum_{v\ll 1}\frac{1}{v^{1-b'}}\frac{(vM_1)^{\frac{1}{2}}}{M_1^{\frac{1}{2}}}\lesssim 1
$$
using the fact that $b'>1/2$.
In general dimensions, this bound becomes $M_2^{\frac{d-2}{2}}$. What we wanted is therefore bounded by
$$
\sum_{M_2\lesssim M_0\sim M_1}\frac{\langle M_2 \rangle^{-s}}{(M_1M_2)^{1-b'}}
M_2^{\frac{d-2}{2}}\prod_{i=0}^2\|f_i\|_{L^2_{\xi,\tau}}.
$$
Since $b'<1$, the sum here when $M_2\gg1$ is bounded as
$$
\sum_{M_2\gg1}\langle M_2 \rangle^{-s}M_2^{b'-1}M_2^{\frac{d-2}{2}}\sum_{M_1\gtrsim M_2}M_1^{b'-1}
\lesssim \sum_{M_2\gg1}  M_2^{2b'-\frac{(2s+6)-d}{2}}
$$
which is summable under the assumption $b'<(2s+6-d)/4$.
But when  $M_2\lesssim1$,
we apply the same procedure as in \textit{Case b} after reducing \eqref{serf} to
\begin{equation*}
\sum_{M_2\lesssim 1\lesssim M_0\sim M_1}\langle M_2 \rangle^{-s}\iiiint\limits_{\xi_0+\xi_1+\xi_2=0\atop \tau_0+\tau_1+\tau_2=0} \frac{f_0(\xi_0,\tau_0)f_1(\xi_1,\tau_1)f_2(\xi_2,\tau_2)}{\langle \tau_1+|\xi_1|^2\rangle^{b}\langle \tau_2\pm|\xi_2|\rangle^{b}} d\xi_1d\xi_2d\tau_1d\tau_2.
\end{equation*}
Then we immediately arrive at
$$
\sum_{M_2\lesssim1\lesssim M_0\sim M_1} \langle M_2\rangle^{-s}M_2^{\frac{d-1}{2}}M_0^{-\frac{1}{2}}
\lesssim\sum_{M_2\lesssim1} M_2^{\frac{d-1}{2}}\sum_{M_0\gtrsim 1}M_0^{-\frac{1}{2}}<\infty.
$$
\qed

Now the case $M = |\tau_0-|\xi_0|^2|$ is complete.
The other cases where $M=|\tau_1+|\xi_1|^2|$ or $M=|\tau_2\pm|\xi_2||$ are handled in the same manner.
One can directly apply the proof above again to the former case, merely exchanging the roles of $(\xi_0,\tau_0)$ and $(\xi_1,\tau_1)$.
The procedure for the latter one is also similar; one needs to bound
\begin{equation*}
\iiiint\limits_{\xi_0+\xi_1+\xi_2=0\atop \tau_0+\tau_1+\tau_2=0} \frac{\langle M_2 \rangle^{-s}f_0(\xi_0,\tau_0)f_1(\xi_1,\tau_1)f_2(\xi_2,\tau_2)}{\langle M\rangle^{1-b'}\langle \tau_0-|\xi_0|^2\rangle^{b}\langle \tau_1+|\xi_1|^2\rangle^{b}} \,d\xi_0\,d\xi_1\,d\tau_0d\tau_1
\end{equation*}
and proceed just as before to arrive at using the change of variables
$u=-\xi_0-\xi_1$ and $v=-|\xi_0|^2+|\xi_1|^2-n-m-\theta_1-\theta_2$.
The determinant of the Jacobian matrix here may differ but it is harmless to the process.
For example,
\begin{equation*}
dudvd\xi_{1,1}\cdots d\xi_{1,j-1}d\xi_{1,j+1}\cdots d\xi_{1,d}=|2\xi_{0,j}+2\xi_{1,j}|d\xi_0d\xi_1=|2\xi_{2,j}|d\xi_0d\xi_1
\end{equation*}
replaces \eqref{jaco}.


\begin{thebibliography}{}

\bibitem{AKS}
J. Ahn, J. Kim and I. Seo,
\textit{On the radius of spatial analyticity for defocusing nonlinear Schr\"odinger equations},
Discrete Contin. Dyn. Syst. 40 (2020), 423-439.

\bibitem{AKS2}
J. Ahn, J. Kim and I. Seo,
\textit{Lower bounds on the radius of spatial analyticity for the Kawahara equation},
Anal. Math. Phys. 11 (2021), 28. 32D15.

\bibitem{B}
J. Bourgain,
\textit{On the Cauchy problem for the Kadomtsev-Petviashvili equation},
Geom. Funct. Anal. 3 (1993), 315-341.

\bibitem{BG}
J. L. Bona and Z. Gruji\'{c},
\textit{Spatial analyticity properties of nonlinear waves},
Math. Models Methods Appl. Sci. 13 (2003), 345-360.

\bibitem{BGK}
J. L. Bona, Z. Gruji\'{c} and H. Kalisch,
\textit{Algebraic lower bounds for the uniform radius of spatial analyticity for the generalized KdV equation},
Ann. Inst. H. Poincar\'{e} Anal. Non Lin\'{e}aire. 22 (2005), 783-797.

\bibitem{BGK2}
J. L. Bona, Z. Gruji\'{c} and H. Kalisch,
\textit{Global solutions of the derivative Schr\"odinger equations in a class of functions analytic in a strip},
J. Differential Equations. 229 (2006), 186-203.

\bibitem{CHT}
J. Colliander, J. Holmer and N. Tzirakis,
\textit{Low regularity global well-posedness for the Zakharov and Klein-Gordon-Schr\"odinger systems},
Trans. Amer. Math. Soc. 360 (2008), 4619-4638.


\bibitem{C}
E. Compann,
\textit{Smoothing for the Zakharov and Klein-Gordon-Schr\"odinger systems on Euclidean spaces},
SIAM J. Math. Anal. 49 (2017), 4206-4231.

\bibitem{ET}
M. B. Erdo\u{g}an and N. Tzirakis,
\textit{Smoothing and global attractors for the Zakharov system on the torus},
Anal. PDE. 6 (2013), 723-750.

\bibitem{FT}
C. Foias and R. Temam,
\textit{Gevrey class regularity for the solutions of the Navier-Stokes equations},
J. Funct. Anal. 87 (1989), 359-369.

\bibitem{HGHL}
C. Huang, B. Guo, D. Huang and Q. Li,
\textit{Global well-posedness of the fractional Klein-Gordon-Schr\"odinger system with rough initial data},
Sci. China Math. 59 (2016), 1345-1366.

\bibitem{HW}
J. Huang and M. Wang,
\textit{New lower bounds on the radius of spatial analyticity for the KdV equation},
J. Differential Equations. 266 (2019), 5278-5317.

\bibitem{K}
Y. Katznelson,
\textit{An Introduction to Harmonic Analysis, corrected ed.},
Dover Publications, Inc., New York (1976).

\bibitem{KM}
T. Kato and K. Masuda,
\textit{Nonlinear evolution equations and analyticity. I},
Ann. Inst. H. Poincar\'{e} Anal. Non Lin\'{e}aire. 3 (1986), 455-467.

\bibitem{P}
 S. Panizzi,
 \textit{On the domain of analyticity of solutions to semilinear Klein-Gordon equations},
Nonlinear Anal. 75 (2012), 2841-2850.

\bibitem{H}
H. Pecher,
\textit{Global solutions of the Klein-Gordon-Schr\"odinger system with rough data},
Differential Integral Equations. 17 (2004), 179-214.

\bibitem{H2}
H. Pecher,
\textit{Low regularity well-posedness for the 3D Klein-Gordon-Schr\"odinger system},
Commun. Pure Appl. Anal. 11 (2012), 1081-1096.

\bibitem{H3}
H. Pecher,
\textit{Some new well-posedness results for the Klein-Gordon-Schr\"odinger system},
Differential Integral Equations. 25 (2012), 117-142.


\bibitem{PS}
G. Petronilho and P. Leal da Silva,
\textit{On the radius of spatial analyticity for the modified Kawahara equation on the line},
Math. Nachr. 292 (2019), 2032-2047.


\bibitem{S}
S. Selberg,
\textit{On the radius of spatial analyticity for solutions of the Dirac-Klein-Gordon equations in two space dimensions},
Ann. Inst. H. Poincar\'{e} Anal. Non Lin\'{e}aire. 36 (2019), 1131-1330.

\bibitem{SS}
S. Selberg and D.O. da Silva,
\textit{Lower bounds on the radius of spatial analyticity for the KdV equation},
Ann. Henri Poincar\'{e}. 18 (2017), 1009-1023.

\bibitem{ST}
S. Selberg and A. Tesfahun,
\textit{On the radius of spatial analyticity for the 1d Dirac-Klein-Gordon equations},
J. Differential Equations. 259 (2015), 4732-4744.

\bibitem{T}
T. Tao,
\textit{Nonlinear Dispersive Equations: Local and Global Analysis},
CBMS Regional Conference Series in Mathematics, vol. 106, American Mathematical Society, Providence, RI, 2006,
published for the Conference Board of the Mathematical Sciences, Washington, DC.

\bibitem{Te}
A. Tesfahun,
\textit{On the radius of spatial analyticity for cubic nonlinear Schr\"odinger equations},
J. Differential Equations. 263 (2017), 7496-7512.

\bibitem{Te2}
A. Tesfahun,
\textit{Asymptotic lower bound for the radius of spatial analyticity to solutions of KdV equation},
Commun. Contemp. Math. 21 (2019), 1850061, 33.
\end{thebibliography}
\end{document}